\documentclass[a4paper;11pt]{amsart}
\usepackage{mathrsfs}\usepackage{graphicx}
 \usepackage[arrow,matrix]{xy}
\usepackage{amsmath,amssymb,amscd,bbm,amsthm,mathrsfs}
\usepackage{color,xcolor}
\usepackage{graphicx}
\usepackage{manfnt}
\newtheorem{thm}{Theorem}[section]
\newtheorem{lem}{Lemma}[section]
\newtheorem{cor}{Corollary}[section]
\newtheorem{prop}{Proposition}[section]
\newtheorem{rem}{Remark}[section]

\theoremstyle{definition}

\begin{document}
\numberwithin{equation}{section}

 \title[Decomposition theorems  for    Hardy spaces  on     products  ]{Decomposition theorems  for    Hardy spaces  on     products of  Siegel upper half spaces and  bi-parameter  Hardy spaces}
\author {Wei Wang${}^\dag$ and Qingyan Wu${}^\ddag$}
\thanks{$  \dag$ Department of Mathematics, Zhejiang University, Zhejiang 310027, China, Email: wwang@zju.edu.cn;}
\thanks{
$ \ddag$ Department of Mathematics,
         Linyi University,
         Shandong  276005, China, Email: qingyanwu@gmail.com}

\begin{abstract}
Products of  Siegel upper half spaces are  Siegel domains, whose Silov boundaries have the structure of products $\mathscr H_1\times\mathscr H_2$ of   Heisenberg groups. By the reproducing formula of  bi-parameter  heat kernel associated to sub-Laplacians, we show that a function in holomorphic Hardy space $H^1$ on such a  domain has boundary value belonging to bi-parameter Hardy space $  H^1 (\mathscr H_1\times \mathscr H_2)$. With the help of atomic decomposition of $  H^1 (\mathscr H_1\times \mathscr H_2)$ and bi-parameter
harmonic analysis, we show that the Cauchy-Szeg\H o projection is a bounded operator from $  H^1 (\mathscr H_1\times \mathscr H_2)$ to  holomorphic Hardy space $H^1$, and any  holomorphic $H^1$ function can be   decomposed as a sum of holomorphic atoms. Bi-parameter atoms on $\mathscr H_1\times\mathscr H_2$ are   more complicated  than $1$-parameter   ones, and so are holomorphic atoms.
\end{abstract}

 \maketitle

\section{Introduction }
Coifman-Rochberg-Weiss \cite{CRW} proved the atomic
decomposition theorem for holomorphic Hardy space $H^1 $ over  the unit
ball in  $ \mathbb{C}^n$.
Garnett-Latter  \cite{GL}  generalized their results to the case $H^p $ for $0 < p< 1$.
Atomic decomposition of holomorphic  $H^p$ functions on   bounded strongly pseudoconvex domains,  pseudoconvex domains of finite type in
$\mathbb{C}^2$  and     convex domains of finite type in
$\mathbb{C}^n$ were
established by Dafni \cite{Da},  Krantz-Li \cite{KL1} \cite{KL2},
Grellier-Peloso  \cite{GP}. Decomposition theorems   of holomorphic  Hardy spaces  have various interesting applications (cf. e.g. \cite{BPS} \cite{KL2} \cite{PV}).

On the other hand, although the bidisc is a simple  Siegel domain   with non-smooth boundary,
boundary behavior of holomorphic functions and  holomorphic
Hardy space  on it  were known to be much more complicated than that on the disc in the late 1970s by Malliavins \cite{MM} and Gundy-Stein \cite{GuS}. It has stimulated  the development of multi-parameter  harmonic
analysis
 since then (cf. e.g. \cite{CF80} \cite{CF85}  \cite{J} \cite{P}). Notably, the definition of a multi-parameter  atom     is more complicated than that of one parameter.  Since a    Siegel
domain  usually has a group of automorphisms including multi-parameter
dilations, it is reasonable to believe that  multi-parameter  harmonic analysis   on  Silov boundaries will  play  an important
role in understanding
boundary behavior of holomorphic functions  and Hardy spaces on such domains.
In this paper,  we consider   products of   Siegel upper half spaces and establish   decomposition theorems  for holomorphic   Hardy spaces on such domains, with the help of
atomic decompositions of  bi-parameter  Hardy spaces   on   products of    Heisenberg groups.

The product of two Siegel upper half spaces is
$  {\mathcal U}:=  {\mathcal U}_1\times   {\mathcal U}_2$, where
\begin{equation}\label{eq:U-alpha}
    {\mathcal U}_\alpha=\left\{(\widetilde  w_\alpha, \widetilde   {\mathbf{z}}_\alpha)\in \mathbb{C} \times\mathbb{C}^{n_\alpha } ;\rho_\alpha(\widetilde
    w_\alpha,  \widetilde
    { \mathbf z}_\alpha):=\operatorname{Im}\widetilde  {w}_\alpha-|\widetilde  {\mathbf z}_\alpha|^2>0 \right\}
 ,\qquad \alpha=1,2,
\end{equation}  are Siegel upper half spaces.
The Silov boundary of $  {\mathcal U}$    is the CR submanifold  defined by $\rho_1=\rho_2=0$.
It is convenient to consider its flat model $\mathscr U:=\mathbb{R}_+^2\times\mathscr H_1\times\mathscr H_2$, where $\mathscr H_\alpha$ is the Heisenberg group,
$\alpha=1,2$. There exists a   quadratic transformation $\pi$ from $
\mathscr U$ to $   {\mathcal U}$. We consider holomorphic  functions on $\mathscr U$   defined in terms of the pulling-back complex structure by $\pi$ (cf. Proposition
\ref{prop:regular-equiv}).
  {\it Holomorphic Hardy space} $H^p(\mathscr U)$ is the space of all holomorphic functions $f $ on $ \mathscr U$ such that
\begin{equation*}
   \|f\|_{H^p(\mathscr U)} :=\left(\sup_{ \mathbf{t} \in \mathbb{R}_+^2}\int_{ \mathbb{R}^2\times\mathbb{C}^{n_1 +n_2}}\left|f\left( \mathbf{s}  +\mathbf{i}
   \mathbf{ t}  , {\mathbf z}_1,  {\mathbf z}_2\right )\right|^pd\mathbf{s} d {\mathbf z}\right)^{\frac 1p}<\infty.
\end{equation*}

The Silov boundary is the product $\mathscr H_1\times\mathscr H_2$ of   Heisenberg groups
with  bi-parameter dilation group.  Recently, the theory of Hardy
spaces
    has been developed on products   of
spaces of homogeneous type \cite{CDLWY} \cite{CFLY}
 \cite{HLL} \cite{HLPW} \cite{HLW}, which include   products of    Heisenberg groups as   special cases.  We need   atomic decomposition of
bi-parameter  Hardy space $  H^1 (\mathscr H_1\times \mathscr H_2)$.
  Let  $\triangle_{\alpha}$ be the sub-Laplacian on $ \mathscr H_\alpha$ and $N$ be a positive integer.
  A function $a \in L^2(\mathscr H_1\times \mathscr H_2)$
 is called  a  {\it $(  2,N)$-atom} if it satisfies the following conditions:
\\
(1) there is an open set $\Omega$ in $\mathscr H_1\times \mathscr H_2$ with finite measure such that ${\rm supp}\, a\subset \Omega$;
\\
(2) $a$ can be further decomposed as
\begin{equation*}
   a=\sum_{R\in m(\Omega)} a_R
\end{equation*}
where $ m(\Omega)$ is the set of all maximal dyadic rectangles contained in $\Omega$, and for each
$R\in m(\Omega)$,  there exists a function $b_R$  belonging
to the domain of $\triangle_1^{ \sigma_1} \otimes \triangle_2^{ \sigma_2} $ in $L^2(\mathscr H_1\times \mathscr H_2)$ for all $\sigma_1, \sigma_2 \in\{0, 1,\ldots,N\}  $,   such
that
\\
(i) $a_R=(\triangle_1^N \otimes \triangle_2^N )b_R$;
\\
(ii)  supp $(\triangle_1^{ \sigma_1} \otimes \triangle_2^{ \sigma_2})b_R\subset  \overline{{C }}R$, where $ \overline{{C }}>1$ is a fixed constant;
\\
(iii) $\|a\|_{L^2(\mathscr H_1\times\mathscr H_2)}\leq |\Omega|^{-\frac 12}$ and
\begin{equation}\label{eq:atom}
   \sum_{R=I\times J\in m(\Omega)}\ell(I)^{4\sigma_1-4N }\ell(J)^{4 \sigma_2-4N }\| ( \triangle_1 ^{\sigma_1} \otimes   \triangle_2 ^{\sigma_2}) b_R\|^2_{L^2(\mathscr H )}\leq
   |\Omega|^{-1}.
\end{equation}

Let $\mathcal P  $  be the    Cauchy-Szeg\H o
projection from $  L^2 (\mathscr H_1\times \mathscr H_2)$ to   holomorphic Hardy space $H^2(\mathscr U )$.
A holomorphic  function $A$ on $\mathscr U $ is called a {\it holomorphic  $(2,N)$-atom} if
there exists a $(2,N)$-atom  $a$ on $\mathscr H_1\times\mathscr H_2$
  such that $A = \mathcal{P}(a)  $.
  {\it Atomic  holomorphic Hardy space} $H^1_{at, N}(\mathscr U)$ is the space of all holomorphic    functions of the form
 $\sum_{j=1}^\infty\lambda_j A_j$ with $\lambda_j\in \mathbb{C},$ $\sum_{j=1}^\infty |\lambda_j| <+\infty,$
 where each $A_j$ is a holomorphic    $(2,N)$-atom and  such a series converges to a holomorphic   function. Moreover, the norm   of $f\in H^1_{at, N}(\mathscr U)$ is the
 infimum of $ \sum_{j=1}^\infty |\lambda_j|  $ taken over all possible decomposition of $f$.

We have the following characterization of   holomorphic Hardy space $  H^1(\mathscr U )$.

\begin{thm}\label{thm:atom} For $N>\max\{n_1+1, n_2+1\}/2$,
$H^1_{at, N}(\mathscr U ){=} H^1(\mathscr U )$  and they have equivalent  norms.
\end{thm}

Holomorphic Hardy space $  H^1 $ on the   Siegel upper half space
was studied by
 Geller \cite{Gel} by using the Beltrami-Laplace operator on   complex hyperbolic  space.   The solution
 formula of  the corresponding  Dirichlet problem reproduces  holomorphic   functions,  playing the role of the Poisson integral in the Euclidean case,  and  can be used to prove     boundary value of a
 holomorphic $H^1$
 function    belonging to the Hardy space $H^1$ on the Heisenberg group \cite{Gel}.  This Dirichlet problem was  generalized by  Graham \cite{Gra} to some modifications of  the Beltrami-Laplace operator.

On the other hand,
Folland-Stein \cite{FS} used the heat kernel to establish the theory of  Hardy spaces  $H^p$ on
 homogeneous groups. As a convenient tool, the heat kernel  has the advantage   that it is a Schwartzian function on a
 homogeneous group such that its convolution
 with a distribution is well defined.
By   identifying the Siegel upper half space $  {\mathcal U}_\alpha$ with $\mathscr U_\alpha=\mathbb{R}_+\times\mathscr H_\alpha $,  we show that a holomorphic
function on $\mathscr U_\alpha $  satisfies    the heat equation associated with the {sub-Laplacian} on the   Heisenberg group.
This phenomenon  was first
observed in \cite{CDLWW}, where    quaternionic regular functions of several variables were proved to satisfy a heat equation on the   flat model   of quaternionic
Siegel upper half space.

\begin{prop}  \label{prop:regular-sublap} A  function $f $ holomorphic on $ \mathscr U=\mathbb{R}_+^2\times\mathscr H_1\times\mathscr H_2$    satisfies the   heat
equations
\begin{equation}\label{eq:regular-sublap}
   \left (\frac \partial{\partial t_\alpha }+  \triangle_{\alpha}\right)f=0,\qquad \alpha=1,2.
\end{equation}
\end{prop}
Bi-parameter  heat kernel reproduces a holomorphic  $H^1$ function on $\mathscr U$, if it is continuous on $\overline{\mathscr U}$ (cf. Proposition \ref{prop:heat-kernel-H1}). This reproducing formula is   more simple and convenient than the solution
 formula  used by
 Geller \cite{Gel}.

If $u$ is a continuous function on $\mathscr U $, define   {\it maximal function} $u  ^*$ on $\mathscr H_1\times \mathscr H_2$ by
 \begin{equation}\label{eq:nontangential}
    u  ^*(\mathbf{g}) := \sup\limits_{(\mathbf{t},\mathbf{ h})\in   \Gamma_{\mathbf{g }} }\left|u(\mathbf{t},\mathbf{ h}) \right|,
 \end{equation}where $\Gamma_{\mathbf{g }}$ is the  {\it non-tangential region} at point $\mathbf{g}=(\mathbf{g }_1,\mathbf{g }_2)\in \mathscr H_1 \times\mathscr H_2$ defined by
\begin{equation}\label{eq:region}
   \Gamma_{\mathbf{g }}:=\left\{(\mathbf{t},\mathbf{h})\in\mathbb{R}_+^2\times\mathscr H_1 \times\mathscr H_2; \|\mathbf{h}_1 ^{-1}  \mathbf{g}_1\|_1^2  <t_1,
   \|\mathbf{h}_2 ^{-1}  \mathbf{g}_2\|_2^2  <t_2\right\}.
\end{equation}
  {\it Bi-parameter  Hardy space $  H^p (\mathscr H_1\times \mathscr H_2)$}
  consists of
  $g\in \mathcal{S}'(\mathscr H_1\times \mathscr H_2)$ such that $ u  ^*\in L^p(\mathscr H_1\times \mathscr H_2)$,  where $u(\mathbf{t},\mathbf{ g})=
  h_\mathbf{t}*g(\mathbf{g})$  and $h_\mathbf{t}$ is the bi-parameter  heat kernel of $e^{-t_1\triangle_1}e^{-t_2\triangle_2}$ with $\mathbf{t}=(t_1,t_2)\in\mathbb{R}_+^2$ and $\mathbf{ g}\in\mathscr H_1\times \mathscr H_2$. It has norm $\|g\|_{H^p (\mathscr H_1\times \mathscr H_2)}=\|u  ^*\|_{L^p  (\mathscr H_1\times \mathscr H_2)}$.

A parabolic version of subharmonicity allows us to show
\begin{prop}  \label{prop:heat-kernel-Hp}  Suppose $f\in H^1(\mathscr U )$   and continuous on $\overline{\mathscr U }$.  Then for $0 < q\leq 1  $, we have
  \begin{equation}\label{eq:heat-rep-p}
   | f ( \mathbf{t},\mathbf{g} )|^q \leq \int_{\mathscr H_1\times \mathscr H_2 }h_\mathbf{t}(\mathbf{h}{}^{-1}  \mathbf{g})|f(\mathbf{0},\mathbf{h}
   )|^qd\mathbf{h}.
 \end{equation}
\end{prop}
 We need the above inequality \eqref{eq:heat-rep-p} for $q<1$ to show that  a holomorphic  $H^1$ function $f$ on $\mathscr U $ satisfies
\begin{equation}\label{eq:nontangential-max}
     f^* \in L^1(\mathscr H_1\times \mathscr H_2).
 \end{equation}This is the most subtle part   as in the classical  theory of   Hardy spaces   for functions annihilated by   generalized
 Cauchy-Riemann operator on $\mathbb{R}_+^{n+1}$ (cf. \cite[Section 3 in Chapter 7]{St70}). Once \eqref{eq:nontangential-max} is proved, the method of
bi-parameter  harmonic analysis  can be applied.
We can show that there exists a boundary distribution $f^b\in \mathcal{S}'(\mathscr
 H_1\times \mathscr H_2)$ in the sense    $f(\varepsilon _1, \varepsilon_2,\cdot)\rightarrow f^b$ in $ \mathcal{S}'(\mathscr
 H_1\times \mathscr H_2
 )$  as $\varepsilon _1, \varepsilon _2\rightarrow 0$,  and
 \begin{equation*}
    f(\mathbf{t}, \mathbf{g})=h_\mathbf{t}* f^b(\mathbf{g})
 \end{equation*}
  in Theorem \ref{thm:boundary-distribution}. Therefore, $f^b\in  H^1
 (\mathscr
 H_1\times \mathscr H_2)$.  Then, we deduce that $\mathcal{P}(a)$ is a  holomorphic  $ H^1(\mathscr U )$ function for each
 boundary $(  2,N)$-atom $a$
 and  the    Cauchy-Szeg\H o projection $\mathcal{P} $  is bounded  from $H^1( \mathscr  H_1\times \mathscr H_2)$  to $H^1(\mathscr U)$. The boundary distribution $f^b$ has  an atomic decomposition
$f^b=
    \sum_{k }  \lambda_k a_k
$ with $\|f^b\| _{H ^1(\mathscr H ) }\approx  \sum_{k } |\lambda_k | $, where each $a_k$ is a      $(2,N)$-atom. At last, we get   holomorphic atomic
 decomposition
\begin{equation*}
f  =\sum   \lambda_k\mathcal P ( a_k)
.
\end{equation*}

This paper is organized as follows. In Section 2, we describe the  flat model  $\mathscr U=\mathbb{R}_+^2\times\mathscr H_1\times\mathscr H_2$ of the product of
Siegel upper half spaces explicitly, and show that a holomorphic  function   on $\mathscr U $ satisfies    the heat equation.  In Section 3, we deduce the
Cauchy-Szeg\H o kernel  on $\mathscr U $, which is the product of Cauchy-Szeg\H o
kernels on $\mathscr U_1 $ and $\mathscr U_2 $, respectively,    from   known formulae for   Cauchy-Szeg\H o kernels on general Siegel domains. Then we show that the Cauchy-Szeg\H o kernel
\begin{equation*}
S((\mathbf{t},\mathbf{g}),\mathbf{g}'),\qquad \mathbf{g} ,\mathbf{g}'\in \mathscr H_1\times\mathscr H_2
  \end{equation*}
  satisfies the condition of rough bi-parameter  Calder\'on-Zygmund kernels   on $ \mathscr H_1\times\mathscr H_2$ uniformly for $\mathbf{t} $, by which we can
  prove $\mathcal P (a)$ belongs to $H^1(\mathscr U )$ for any   $(  2, N)$-atom $a$ on the  group $\mathscr H_1\times \mathscr H_2$ in Section 4. Consequently,
  the    Cauchy-Szeg\H o
projection  $\mathcal P  $  is   bounded   from $  H^1 (\mathscr H_1\times \mathscr H_2)$ to   holomorphic Hardy space $H^1(\mathscr U )$.
In Section 5, an
$H^1(\mathscr U )$ function is proved to have a boundary distribution $f^b\in  H^1 (\mathscr H_1\times \mathscr H_2)$. It  has atomic decomposition, which under the action of
the    Cauchy-Szeg\H o
projection  $\mathcal P  $ gives us holomorphic  atomic decomposition.
We use   parabolic  maximum principle and   parabolic version of subharmonicity of $|f|^p$  to prove Proposition \ref{prop:heat-kernel-Hp} in Section 6.

\section{ the  Flat model}\subsection{ The  flat model $\mathscr U $}Let $\mathscr H_\alpha$ be the Heisenberg group $\mathbb{R}\times\mathbb{C}^{ n_\alpha}$.
We write a point of $\mathscr H_\alpha$ as    $\mathbf{g}_\alpha := (s_\alpha , \mathbf{z}_\alpha)  $ with $s_\alpha\in \mathbb{R}$ and
$
   \mathbf{z}_\alpha=({z}_{\alpha 1},\ldots,{z}_{\alpha n_\alpha})\in\mathbb{C}^{ n_\alpha}$, where  ${z}_{\alpha j}= x_{\alpha j}+\mathbf{i}x_{\alpha(n_\alpha+j)}$, $j=1,\ldots, n_\alpha$.
Its
multiplication is  given by
\begin{equation}\label{eq:multiplication}
   (s_\alpha , \mathbf{z}_\alpha) (s_\alpha ', \mathbf{z}_\alpha')= \left(s_\alpha +s_\alpha '+ 2 {\rm Im} \langle \mathbf{z}_\alpha,  \mathbf{z}_\alpha
   '\rangle,\mathbf{z}_\alpha +\mathbf{z}_\alpha'\right),
\end{equation}where $\langle \cdot ,  \cdot\rangle$ is the standard Hermitian inner product on $\mathbb{C}^{ n_\alpha}$.

The identity element of $\mathscr H_\alpha$ is the origin $\mathbf{0}_\alpha$, and the inverse element of $\mathbf{g}_\alpha =(  s_\alpha, \mathbf{z}_\alpha)$ is
$\mathbf{g}_\alpha^{-1} =( - s_\alpha, -\mathbf{z}_\alpha)$. The {\it homogeneous norm} $\|\cdot\|_\alpha$ of $\mathbf{g}_\alpha$ is defined by
$$ \|\mathbf{g}_\alpha\|_\alpha:= (| \mathbf{z} _\alpha  |^4+| s_\alpha|^2 )^{\frac 1   4}.$$
  Then,  $ \| \mathbf{h}_\alpha ^{-1}  \mathbf{g}_\alpha\|_\alpha  $ is a distance between  $\mathbf{g}_\alpha, \mathbf{h} _\alpha\in \mathscr H_\alpha $.  We define  balls  in $\mathscr H_\alpha$ by
  \begin{equation*}
     B_\alpha(\mathbf{g}_\alpha,r):=\left\{\mathbf{h}_\alpha \in \mathscr
  H_\alpha; \| \mathbf{h}_\alpha ^{-1}  \mathbf{g}_\alpha\|_\alpha <r \right\}.
  \end{equation*}
   The   Heisenberg group $\mathscr H_\alpha $ is a homogeneous group with dilations
 $\delta_r^{(\alpha)}(s_\alpha, \mathbf{z}_\alpha)=(r^2s_\alpha, r \mathbf{z}_\alpha  ),$ $  r>0.$ The Lebesgue measure $d\mathbf{g}_\alpha$  is an invariant measure   on
 $\mathscr H_\alpha $.
Then for any measurable set $E\subset\mathscr H_\alpha$,
 $|\delta_r^{(\alpha)}(E)|=r^{Q_\alpha}|E|,$
 where $Q_\alpha=2n_\alpha+2$ is the {\it homogeneous dimension} of $\mathscr H_\alpha$.
Denote by $\tau_{\mathbf{h}_\alpha}$ the {\it left translation } on $\mathscr H_\alpha$ by $\mathbf{h}_\alpha$, i.e.
$
   \tau_{\mathbf{h}_\alpha}(\mathbf{g}_\alpha)=\mathbf{h}_\alpha  \mathbf{g}_\alpha.
$

 The product
$
   \mathscr H: =\mathscr H_1\times \mathscr H_2
$
is a nilpotent Lie group of step two.
We write a point   in $\mathscr H_1\times \mathscr H_2 $ as $  \mathbf{g }=(  \mathbf{g}_1,\mathbf{g}_2) $ with   $\mathbf{g}_\alpha\in \mathscr H_\alpha$,
 $\alpha=1,2$. It has bi-parameter dilations
 $\delta_{\mathbf{r}}(  \mathbf{g}_1,\mathbf{g}_2):=(  \delta_{r_1}^{(1)}(\mathbf{g}_1),\delta_{r_2}^{(2)}(\mathbf{g}_2)) $, where $  \mathbf{r}=(r_1,r_2) \in \mathbb{R}^2_+ $, and left translation
 $\tau_{\mathbf{h}} (\mathbf{g} ):=(\tau_{\mathbf{h}_1}(\mathbf{g}_1), \tau_{\mathbf{h}_2}(\mathbf{g}_2))$ for $\mathbf{h}= (  \mathbf{h}_1,\mathbf{h}_2)$,
 $\mathbf{g}= (  \mathbf{g}_1,\mathbf{g}_2)\in\mathscr H_1\times
 \mathscr H_2$. The Lebesgue measure $d\mathbf{g}  =d\mathbf{g}_1d\mathbf{g}_2$ is also an invariant measure on $\mathscr H_1\times
 \mathscr H_2$. The convolution of two functions $u$ and $v$ on it is defined as
 \begin{equation*}
    u*v(\mathbf{g}):=\int_{\mathscr H_1\times \mathscr H_2} u(\mathbf{h}^{-1}\mathbf{g})v(\mathbf{h})d\mathbf{h} .
 \end{equation*}

 We   write a point in $\mathscr U_\alpha$ as $  ( t_\alpha ,\mathbf{g}_\alpha)$,
 $\alpha=1,2$,  with   $t_\alpha\in \mathbb{R}_+$, $\mathbf{g}_\alpha=(  s_\alpha, \mathbf{z}_\alpha)\in \mathscr
 H_\alpha$. It is  also convenient to use complex coordinate  $(w_\alpha,  \mathbf{z}_\alpha)$ for  a point in $\mathscr U_\alpha$, where
$w_\alpha=s_\alpha+\mathbf{i}t_\alpha $.
We can   identify    $  {\mathcal U}_\alpha$ with $\mathscr U_\alpha $ by    quadratic transformation $\pi_\alpha: \mathscr
U_\alpha =\mathbb{R}_+ \times\mathscr H_\alpha  \rightarrow   {\mathcal U}_\alpha   $ given by
 \begin{equation}    \label{eq:pi-alpha}
 \begin{split}
(w_\alpha,  \mathbf{z}_\alpha)&\mapsto (\widetilde  w_\alpha,  \widetilde  {\mathbf z}_\alpha)=\left (w_\alpha+\mathbf{i}|\mathbf{z}_\alpha|^2,
\mathbf{z}_\alpha\right) .
 \end{split} \end{equation}
Therefore, we can    identify    $  {\mathcal U}$ with $\mathscr U $ by    quadratic transformation
$
   \pi=\pi_1\times \pi_2:\mathscr U=\mathbb{R}_+^2\times\mathscr H_1\times\mathscr H_2\rightarrow  {\mathcal U}
$,
 \begin{equation}\label{eq:pi}
   (  { \mathbf  w} ,  { \mathbf z} )\mapsto (\widetilde  { \mathbf  w} ,\widetilde  { \mathbf z} )=(  w_1+\mathbf{i}|  {\mathbf z}_1|^2,  w_2+\mathbf{i}|
   {\mathbf
   z}_2|^2, { \mathbf z} ) .
  \end{equation}Here and in the sequel, we write   a point in $\mathscr U$ as $( {\mathbf{w}}, {\mathbf{z}})$ with ${\mathbf{w}} : =   ( {w}_1, {w}_2), {\mathbf{z}}:=( {\mathbf{z}}_1, {\mathbf{z}}_2) $  or $ (\mathbf{t } ,\mathbf{g })=( t_1,t_2 ,\mathbf{g}_1,\mathbf{g}_2)$.  For an object  on $\mathcal U$,
 we add tilde to the notation   corresponding to that on  $\mathscr U$.
  Let
 \begin{equation}\label{eq:dbar-w}
    \frac \partial{\partial \overline w_\alpha}=\frac 12\left (\frac \partial{\partial s_\alpha }+\mathbf{i} \frac \partial{\partial t_\alpha }\right) \qquad
\text{ and }
   \qquad
    \frac \partial{\partial \overline{z}_{\alpha j}}=\frac 12\left (\frac \partial{\partial x_{\alpha j}}+\mathbf{i} \frac {\partial\hskip 8mm}{\partial
    x_{\alpha(n_\alpha+j)}}\right).
 \end{equation}

Note that
${\rm Im} \langle \mathbf{z}_\alpha,  \mathbf{z}_\alpha
   '\rangle= \sum_{j=1}^{n_\alpha}(-x_{\alpha j} x'_{\alpha(n_\alpha+j)} + x_{\alpha(n_\alpha+j)}x'_{\alpha j})$.
Then
 \begin{equation*}\begin{split}
    X_{\alpha j} =\frac \partial{\partial x_{\alpha j}}+2  x_{\alpha(n_\alpha+j)}\frac \partial{\partial s_\alpha }, \qquad \qquad   X_{\alpha(n_\alpha+j)} =\frac
   {\partial\hskip 8mm}{\partial
    x_{\alpha(n_\alpha+j)}}-2  x_{\alpha j}\frac \partial{\partial s_\alpha },
 \end{split}  \end{equation*}$j=1,\ldots,n_\alpha$,  are left invariant vector fields on the group $\mathscr H_\alpha$,
and
 \begin{equation} \label{eq:brackets}
   \left [X_{\alpha j},X_{\alpha(n_\alpha+j)}\right]=-4 \frac \partial{\partial s_\alpha },
 \end{equation}
 and all other brackets vanish. The {\it sub-Laplacian} on  the Heisenberg group  $ \mathscr H_\alpha$ is
$
  \triangle_{\alpha}:= -\frac  1{4n_\alpha}\sum_{j=1}^{2n_\alpha}   X_{\alpha j} ^2
$. Denote $
   \overline{Z}_{\alpha j}=\frac 12  (X_{\alpha j}+\mathbf{i}X_{\alpha(n_\alpha+j)})
$. Then\begin{equation*}
  \overline{Z}_{\alpha j}  =\frac \partial{\partial \overline{z}_{\alpha j}}- \mathbf{i} {z}_{\alpha j}\frac \partial{\partial s_\alpha }.
    \end{equation*}

The following proposition  characterizes  the complex
structure pulled back by $\pi$. Namely, a function $f$ is {\it holomorphic} on $\mathscr U$  with respect to this complex
structure if and only if \eqref{eq:regular} is satisfied.

  \begin{prop}  \label{prop:regular-equiv} A  function $\widetilde {f}$ is  holomorphic on $  {\mathcal U}  $ if and only if $f:=\pi^*
  \widetilde {f}
  $   satisfies
  \begin{equation}\label{eq:regular}
    \frac {\partial f}{\partial \overline{w}_\alpha }=0   \qquad {\rm and }\qquad  \overline{Z}_{\alpha j} f=0,
  \end{equation}on $ \mathscr U $,  where $j=1,\ldots,n_\alpha,$ $ \alpha=1,2 $, and $(\pi^*
  \widetilde {f})(  { \mathbf  w} ,  { \mathbf z} )=\widetilde {f}(\pi (  { \mathbf  w} ,  { \mathbf z} ))$.
  \end{prop}
\begin{proof} Recall that for a vector field $X$ on $\mathscr U$,  the pushing forward vector field $
\pi_*X $ on $ {\mathcal U}  $ is defined by
\begin{equation*}
  \left.  (\pi_*X)\psi\right|_{\pi (  { \mathbf  w} ,  { \mathbf z} )}= X[\psi(\pi (  { \mathbf  w} ,  { \mathbf z} ))]
\end{equation*}
for any scalar function $\psi$ on $ {\mathcal U}$. If we write coordinates of $\mathcal{U}$ as $(\widetilde{\mathbf{w}},\widetilde{\mathbf{z}})$, where $ \widetilde{\mathbf{w}}=(  {w}_1, {w}_2)$, $ \widetilde{\mathbf{z}}=(\widetilde{\mathbf{z}}_1,\widetilde{\mathbf{z}}_2)$, and  $\widetilde{w}_\alpha=\widetilde  s_\alpha+\mathbf{i}\widetilde  t_\alpha$, $
   \widetilde{\mathbf{z}}_\alpha=(\widetilde{{z}}_{\alpha 1},\ldots,\widetilde{{z}}_{\alpha n_\alpha})  $, $\widetilde  z_{\alpha j}=\widetilde  x_{\alpha j}+\mathbf{i}\widetilde  x_{\alpha (n_\alpha+j)}$, then the transformation $\pi$ in \eqref{eq:pi} is given by
\begin{equation*}
   \widetilde  s_\alpha= s_\alpha  ,  \qquad \widetilde  t_\alpha= t_\alpha+\sum_{j= 1}^{2n_\alpha}  |x_{\alpha j}|^2
, \qquad \widetilde  x_{\alpha j}= x_{\alpha j}.
\end{equation*}
  It is
direct to check that
\begin{equation}\label{eq:pi-XY}\begin{split}\pi_*\frac {\partial}{\partial s_\alpha}& =\frac {\partial} {\partial \widetilde  s_\alpha },\qquad\qquad
    \pi_*\frac {\partial}{\partial t_\alpha} =\frac {\partial}{ \partial\widetilde   t_\alpha },
    \\ \pi _*\frac {\partial}{ \partial x_{\alpha j}}&=\frac {\partial}{\partial\widetilde   x_{\alpha j}}+ 2\widetilde  x_{\alpha j}
\frac {\partial}{\partial \widetilde  t_\alpha},\qquad j= 1, \cdots,2n_\alpha.
\end{split}\end{equation}
Consequently, we have
\begin{equation}\label{eq:pi-Z0}\begin{aligned}\pi_*\overline{Z}_{\alpha j}=&
\pi_*\left(\frac \partial{\partial \overline{z}_{\alpha j}}- \mathbf{i} {z}_{\alpha j}\frac \partial{\partial s_\alpha }\right)\\=&\frac 12 \left(
\frac \partial{\partial \widetilde  x_{\alpha j}}+\mathbf{i} \frac {\partial\hskip 8mm}{\partial \widetilde   x_{\alpha(n_\alpha+j)}}\right)+ (\widetilde   x_{\alpha j}  +
\mathbf{i}\widetilde
 x_{\alpha(n_\alpha+j)})\frac  \partial {\partial \widetilde  t_\alpha}-\mathbf{i} {\widetilde  z}_{\alpha j}\frac \partial{\partial\widetilde   s_\alpha }
  \\ =& \frac \partial
 {\partial  \overline{\widetilde  {z}}_{\alpha j}}-2\mathbf{i}  {\widetilde  z}_{\alpha j}\frac \partial
 {\partial \overline{{\widetilde  w}}_\alpha}
,
\end{aligned}\end{equation}where $\frac \partial{\partial \overline {\widetilde  w}_\alpha}=\frac 12\left (\frac \partial{\partial  \widetilde  s_\alpha
}+\mathbf{i} \frac
\partial{\partial  \widetilde    t_\alpha }\right)$. Thus,
\begin{equation*}\begin{split}
\left.
 \overline{Z}_{\alpha j}(\pi^* \widetilde {f})\right|_{(  { \mathbf  w} ,  { \mathbf z} )}& = \left. \pi_*\overline{Z}_{\alpha j}  \widetilde {f} \right|_{\pi(  {
 \mathbf  w} ,  { \mathbf z} )}= \left.\left(\frac {\partial\widetilde {f}}
 {\partial  \overline{\widetilde  {z}}_{\alpha j}}-2\mathbf{i}  {\widetilde  z}_{\alpha j}\frac {\partial\widetilde {f}}
 {\partial \overline{{\widetilde  w}}_\alpha}\right)\right|_{\pi(  { \mathbf  w} ,  { \mathbf z} )},\\
 \left.\frac {\partial(\pi^*  \widetilde f)}{\partial \overline {   w}_\alpha} \right|_{ (  {
 \mathbf  w} ,  { \mathbf z} )}&= \left. \left(\pi_*\frac
{\partial}{\partial \overline {  w}_\alpha}\right) \widetilde f\right|_{\pi(  { \mathbf  w} ,  { \mathbf z} )}  = \left.\frac
{\partial\widetilde f}{\partial \overline {\widetilde  w}_\alpha} \right|_{\pi(  { \mathbf  w} ,  { \mathbf z} )}
.
\end{split}\end{equation*} We see that
$\widetilde {f}$ is holomorphic on $  {\mathcal U}  $, i.e. $\frac {\partial\widetilde {f}}
 {\partial  \overline{\widetilde  {z}}_{\alpha j}}=0$, $\frac {\partial\widetilde {f}}
 {\partial \overline{{\widetilde  w}}_\alpha}=0$,  if and only if
\eqref{eq:regular} holds for $f=\pi^* \widetilde {f}$.
\end{proof}
\begin{rem}  $ \pi_*\overline{Z}_{\alpha j}$ is a vector field tangential to the    boundary   $ \partial  {\mathcal U}_\alpha$,  since
 $  \frac {\partial \rho_\alpha}
 {\partial  \overline{\widetilde  {z}}_{\alpha j}}-2\mathbf{i}  {\widetilde  z}_{\alpha j}\frac {\partial \rho_\alpha}
 {\partial \overline{\widetilde  {w}}_\alpha} =0,$ $j=1,\ldots,n_\alpha ,$ $\alpha=1,2$, by definition.
\end{rem}

Holomorphic Hardy space $H^p(  {\mathcal U})$ is the space of all holomorphic functions $\widetilde{f }$ on $   {\mathcal U}$ such that
\begin{equation*}
   \|\widetilde  f\|_{H^p(  {\mathcal U})} =\left(\sup_{\widetilde  t_1,\widetilde  t_2>0}\int_{ \mathbb{R}^2\times\mathbb{C}^{n_1 +n_2}}\left|\widetilde
   f\left(\widetilde  s_1
   +\mathbf{i}(\widetilde  t_1 + |\widetilde  {\mathbf z}_1|^2),\widetilde  s_2 +\mathbf{i}(\widetilde  t_2 + |\widetilde  {\mathbf z}_2|^2),\widetilde  {\mathbf
   z}_1,\widetilde  {\mathbf
   z}_2\right )\right|^pd\widetilde s_1d\widetilde s_2 d\widetilde  {\mathbf z}\right)^{\frac 1p}<\infty.
\end{equation*}
The diffeomorphism $\pi$ in (\ref{eq:pi}) induces an   isomorphism of Hardy spaces $\pi^*:
   H^p(  {\mathcal{U}} ) \longrightarrow   H^p(\mathscr U) $ given by
\begin{equation}\label{eq:isomorphism}\begin{split}\widetilde {f} &\mapsto\left (\pi^*\widetilde {f}\right)(\mathbf{w} , \mathbf{z} ):=\widetilde
{f}(w_1+\mathbf{i}|\mathbf{z}_1|^2,w_2+\mathbf{i}|\mathbf{z}_2|^2, \mathbf{z}_1,\mathbf{z}_2 ),
\end{split}\end{equation}
with $\|\cdot\|$ preserved.

\begin{prop}\label{lem:bound}
    There exists a positive constant $C$ depending only on $Q_1$, $Q_2$ and $  p >0  $
  such that
  \begin{equation}\label{eq:bound}
    |f ( \mathbf{t } , \mathbf{g})|\leq  C\|f \|_{H^p( \mathscr U ) }t_1^{-\frac {Q_1}{2p }}t_2^{-\frac {Q_2}{2p }},
  \end{equation}
  for any $f\in H^p( \mathscr U ) $  and any $(\mathbf{t}  , \mathbf{g})\in \mathscr U$.
 \end{prop}
 \begin{proof}
 Note that if $f\in H^p( \mathscr U )$, then $ \widehat{{f}}(\mathbf{t}, \mathbf{h}):=f(\mathbf{t},\mathbf{g}  \mathbf{h})$ for fixed $\mathbf{g}$ is also
 holomorphic on $  \mathscr U$, since
\begin{equation*}
   \overline Z_{\alpha j}  \widehat{{f}}( \mathbf{t} ,\mathbf{h } )=  (\overline Z_{\alpha j}  f)( \mathbf{t},\mathbf{g}  \mathbf{h}  )=0  , \quad\qquad {\rm
   and }\qquad \frac {\partial \widehat{ {f}} }{\partial \overline{w}_\alpha }( \mathbf{t} ,\mathbf{h } )=\frac {\partial f }{\partial \overline{w}_\alpha }(
   \mathbf{t},\mathbf{g}  \mathbf{h}  )=0,
\end{equation*}by   holomorphicity of $f$ on $ \mathscr U
$ and left invariance of $ \overline Z_{\alpha j} $  and $\frac {\partial   }{\partial \overline{w}_\alpha }$. Also, we have
\begin{equation*}
   \int_{\mathscr H_1\times \mathscr H_2}| \widehat{{f}}( \mathbf{t} , \mathbf{h} )| d\mathbf{h}=
\int_{\mathscr H_1\times \mathscr H_2}| {f}( \mathbf{t}, \mathbf{h}  )| d\mathbf{h},
\end{equation*}
by the invariance of the measure $d\mathbf{h}$. We get $ \widehat{{f}}\in H^1( \mathscr U )$. Hence,  it is sufficient to prove (\ref{eq:bound}) for $\mathbf{g}=\mathbf{0}$.

To apply the mean value formula, we need to transform $f$ to a holomorphic function $ \widetilde{ f}$ on $ \mathcal U  $ in the usual sense.
Recall that
 \begin{equation}\label{eq:tilde-f-f}
    \widetilde {f}(  \widetilde {{\mathbf w} } ,  \widetilde  {\mathbf z} ):=f(\widetilde  w_1-\mathbf{i}|\widetilde {\mathbf{ z}}_1|^2 ,\widetilde
    w_2-\mathbf{i}|\widetilde {\mathbf{
    z}}_2|^2 , \widetilde{\mathbf{z}})
 \end{equation}
   belongs to $ H^1(\mathcal{U} )$ with the same norm by (\ref{eq:isomorphism}).    Let $D ( z, r)$ be the disc in $ \mathbb{C} $ with radius $r$ and center $z$.
   Since the polydisc
$D_\alpha:=D (\mathbf{i}t_\alpha, t_\alpha/2)\times D (0, \sqrt{t_\alpha/4n_\alpha} )\times\cdots\times D (0, \sqrt{t_\alpha/4n_\alpha} )) \subset   {\mathcal U}_\alpha$,
we have
   \begin{equation*}
      \widetilde  f (\mathbf{i}\mathbf{t} , \mathbf{0})=\frac 1{|D_{1} || D_{2}| } \int_{D_{1}}\int_{ D_{2}} \widetilde  f (\widetilde  {\mathbf w},\widetilde  {\mathbf
      z})d\widetilde {\mathbf
      w}d\widetilde  {\mathbf z},
   \end{equation*}by the mean value formula, where $\mathbf{t}=(t_1,t_2)\in \mathbb{R}^2_+$  and
 \begin{equation*}
    D_{\alpha}  \subset\{ (\widetilde  w_\alpha,\widetilde  {\mathbf z}_\alpha) \in   {\mathcal U}_\alpha; t_\alpha/4< \operatorname{Im}
 \widetilde
 w_\alpha-|\widetilde  {\mathbf z}_\alpha|^2< 2t _\alpha\}
 \end{equation*}
   by definition. So we have
   \begin{align*}
     |\widetilde  f(\mathbf{i}\mathbf{t},\mathbf{0})| &\leq  \frac {C}{t _1^{ n_1+2}t _2^{ n_2+2}} \int_{ t _1/4< \operatorname{Im}\widetilde   w_1-|\widetilde  {\mathbf
     z}_1|^2< 2t _1
     }   \int_{  t_2/4< \operatorname{Im} \widetilde  w_2-|\widetilde  {\mathbf z}_2|^2< 2t_2 }  \left|\widetilde  f(\widetilde  {\mathbf w},\widetilde  {\mathbf
     z})\right|  d\widetilde
     {\mathbf w}d\widetilde  {\mathbf z}
      \\& = \frac {C}{t _1^{ n_1+2}t _2^{ n_2+2}}    \int_{\operatorname{Im}\widetilde   w_1\in(  {t _1}/4,2t_1 )}\int_{\operatorname{Im} \widetilde
      w_2\in(  {t
      _2}/4,2t_2 )}\int_{ \mathbb{R}^{2} \times \mathbb{C}^{n_1} \times \mathbb{C}^{n_2}} \left|\widetilde  f\left(\widetilde  w_1+\mathbf{i}|\widetilde  {\mathbf
      z}_1|^2,\widetilde   w_2+\mathbf{i}|\widetilde  {\mathbf z}_2|^2, \widetilde  {\mathbf z}\right)\right|  d\widetilde  {\mathbf w}d\widetilde  {\mathbf z}
      \\
    & = \frac {C}{t_1^{ n_1+2}t _2^{ n_2+2}}   \int_{t_1/4}^{  2t _1  }dt_1\int_{t _2/4}^{  2t_2  }dt_2\int_{ \mathscr H_1\times \mathscr H_2} |f (
    \mathbf{t},\mathbf{g})|  d \mathbf{g}
     \\&\leq \frac {4C}{t _1^{ n_1+1}t _2^{ n_2+1}} \|f\|_{H^1(\mathscr U )} ,
\end{align*}by \eqref{eq:tilde-f-f},
  where    we have used the coordinates transformation $(\widetilde  {\mathbf w},\widetilde  {\mathbf z} )\mapsto (\widetilde  w_1+\mathbf{i}|\widetilde  {\mathbf
  z}_1|^2,\widetilde
  w_2+\mathbf{i}|\widetilde  {\mathbf z}_2|^2, \widetilde  {\mathbf z} )$ in the first identity,
 which obviously preserves the volume form. The estimate follows from $ \widetilde  f  (\mathbf{i}\mathbf{t},\mathbf{0})=f(\mathbf{i}\mathbf{t },\mathbf{0})$ by definition \eqref{eq:tilde-f-f}.
  \end{proof}
  \subsection{ The  heat equations}  The   {\it heat operator} on $\mathscr H_\alpha$ is
\begin{equation*}
   \mathcal{L}_\alpha:=\frac  1{4n_\alpha}\sum_{j=1}^{2n_\alpha}   X_{\alpha j} ^2 -\frac \partial{\partial t_\alpha }.
\end{equation*}
  \begin{proof}[Proof of Proposition \ref{prop:regular-sublap}] Note that
   \begin{equation*}\begin{split}
  4\sum_{j=1}^{n_\alpha}{Z}_{\alpha j}\overline{Z}_{\alpha j} & =\sum_{j=1}^{n_\alpha}\left (X_{\alpha j}-\mathbf{i}X_{\alpha(n_\alpha+j)}\right)
 \left (X_{\alpha j}+\mathbf{i}X_{\alpha(n_\alpha+j)}\right)\\&=\sum_{j=1}^{2n_\alpha}
     X_{\alpha j} ^2 +\mathbf{i}\sum_{j=1}^{n_\alpha}\left[X_{\alpha j},X_{\alpha(n_\alpha+j)}\right] =\sum_{j=1}^{2n_\alpha}    X_{\alpha j} ^2
  -4n_\alpha\mathbf{ i}\frac \partial{\partial s_\alpha }
 \end{split} \end{equation*}by  brackets in \eqref{eq:brackets}.
 Thus
   \begin{equation*}\begin{split} 4n_\alpha\mathcal{L}_\alpha f&=\sum_{j=1}^{2n_\alpha}   X_{\alpha j} ^2  f-4n_\alpha \frac
   {\partial f}{\partial t_\alpha }
   =4\sum_{j=1}^{n_\alpha}{Z}_{\alpha j}\overline{Z}_{\alpha j}f +4n_\alpha\mathbf{ i}\frac  {\partial f}{\partial s_\alpha }-4n_\alpha \frac  {\partial f}{\partial
   t_\alpha }    \\
   & =
    4\sum_{j=1}^{n_\alpha}{Z}_{\alpha j}\overline{Z}_{\alpha j}f +8n_\alpha \mathbf{i}\frac {\partial f}{\partial \overline{w}_\alpha }=0,
 \end{split} \end{equation*}by
 the expression of  $ \frac \partial{\partial \overline w_\alpha}  $ in \eqref{eq:dbar-w} and  Proposition \ref{prop:regular-equiv}, since $f $ is holomorphic on $ \mathscr U$.
\end{proof}

Let $h_{t_\alpha}^{(\alpha)}(\mathbf{g})$ be the heat kernel of $e^{-t_\alpha\triangle_\alpha }$ on $\mathscr H_\alpha$. Then,
$
   h_{\mathbf{t } }(\mathbf{g}_1,\mathbf{g}_2)=h_{t_1}^{(1)}(\mathbf{g}_1)h_{t_2}^{(2)}(\mathbf{g}_2)
$
  is the bi-parameter   heat kernel of
$e^{-t_1\triangle_1 }e^{-t_2\triangle_2 }$, where $\mathbf{t }  = (t_1,t_2)\in \mathbb{R}_+^2$.
\begin{prop}  \label{prop:heat-kernel-H1}  Suppose $f\in H^p(\mathscr U )$ with $p\geq 1$ and $f$ is continuous on $\overline{\mathscr
 U }$. Then
  \begin{equation}\label{eq:heat-kernel-H1}
 {f} (\mathbf{t} , \mathbf{g})=\int_{\mathscr H_1\times \mathscr H_2}h_{\mathbf{t } }(\mathbf{h}{}^{-1}  \mathbf{g})f(\mathbf{0},\mathbf{h}) d \mathbf{h}  .
 \end{equation}
The formula also holds for $f\in H^p(\mathscr U_\alpha )$, $\alpha=1,2$.
\end{prop}
The reproducing formulae for $p= 1,2$ will be used.
We will use   parabolic  maximum principle and   parabolic version of subharmonicity of $|f|^p$   to prove this proposition   and
Proposition \ref{prop:heat-kernel-Hp} in Section \ref{subharmonicity}.

 \section{   The  Cauchy-Szeg\H o kernel  and associated integral operators  }
 \subsection{  The  Cauchy-Szeg\H o kernel } Let us  deduce the Cauchy-Szeg\H o kernel  on $   \mathscr U $ from   known formulae for Cauchy-Szeg\H o kernels on general Siegel domains \cite{KS}.

Let $\Omega\subset    \mathbb{R}^{m}$ be a regular cone, i.e. it is a nonempty open convex  with vertex at $0$ and
containing no entire straight line. The dual cone  $\Omega^*$ is the set of all $\lambda\in(\mathbb{R}^{m})^*$  such that $\langle\lambda, x\rangle > 0$ for all
 $x\in \overline{\Omega} \setminus\{0\}$. Given a regular cone $\Omega\subset    \mathbb{R}^{m}$, we say that an Hermitian form
 $\Phi: \mathbb{C}^{n }\times \mathbb{C}^{n }\rightarrow \mathbb{C}^{m}$ is {\it $\Omega$-positive} if $\Phi(z ,z )\in \overline{\Omega}$ for any $z
 \in\mathbb{C}^{n }$ and $\Phi(z ,z )=0$ only if $z =0$. The  domain
\begin{equation*}
   {\mathcal D }: = \left\{ \zeta= (\zeta',\zeta'') \in \mathbb{C}^{m}\times \mathbb{C}^{n };  \operatorname{Im} \zeta' - \Phi(\zeta'',\zeta'' )\in \Omega\right\}
\end{equation*}
is called the {\it  Siegel domain determined by $\Phi$ and $\Omega$}. Its {\it Silov boundary} ${\mathcal S }$ is the CR submanifold defined by the equation
$
  \operatorname{Im} \zeta' - \Phi(\zeta'',\zeta'' )=0 ,
$ which has the structure of   a nilpotent Lie group  of step two.

Holomorphic  Hardy space
$H^2(\mathcal D )$ consists of all holomorphic functions $f $ on $\mathcal D $ such that
\begin{equation}\label{eq:H2}
   \|f\|_{H^2(\mathcal D )}^2=\sup_{y \in \Omega}\int_{\mathbb{R}^{m}\times \mathbb{C}^{n }}|f(x +\mathbf{i}y +\mathbf{i}\Phi(\zeta''  ,\zeta''  ),\zeta''  )|^2dx
   d\zeta'' <\infty.
\end{equation}
The  Cauchy-Szeg\H o projection $\mathcal P  $ from $L^2(\mathcal{S})$ to
$H^2(\mathcal D )$ has a reproducing kernel $S(\zeta, \eta)$, the {\it Cauchy-Szeg\H o kernel}, which is holomorphic in $\zeta\in \mathcal D $ and anti-holomorphic
in $\eta\in\mathcal D $. Namely, for $f\in H^2(\mathcal D )$, we have
\begin{equation}\label{eq:Szego-projection}
   f(\zeta)=\int_{{\mathcal S}} S(\zeta,\eta)f(\eta)d\beta(\eta),
\end{equation}
where $d\beta$ is the measure corresponding to $dx
   d\zeta''$ in \eqref{eq:H2}.

For $\lambda\in \Omega^*$, denote $B_\lambda(\zeta'',\eta'') :=4\langle\lambda,\Phi(\zeta'',\eta'')\rangle$,  an Hermitian form on $\mathbb{C}^{n }$, whose
associated Hermitian matrix is also denoted by  $B_\lambda$. The explicit formula for $S(\zeta, \eta)$ \cite[Theorem 5.1]{KS} is known
 as
\begin{equation}\label{eq:Szego}
   S(\zeta, \eta)=\int_{\Omega^*}e^{-2\pi\langle\lambda,\rho(\zeta,\eta)\rangle}\det B_\lambda\, d\lambda,
\end{equation}
for $\zeta=(\zeta',\zeta'')\in \mathcal{D}, \eta=(\eta',\eta'')\in  \mathcal{S}$, where
\begin{equation*}
   \rho(\zeta,\eta)=\mathbf{i}\left(\overline{\eta'}-\zeta'\right)-2\Phi(\zeta'' ,\eta'').
\end{equation*}

The product $  {\mathcal U}  $ of two Siegel upper half spaces ${\mathcal U}_1$ and $  {\mathcal U}_2$ in \eqref{eq:U-alpha} is a  Siegel domain   with the cone $\Omega=
\mathbb{R}^{2}_+\subset    \mathbb{R}^{2}$, $m=2$, $n=n_1+n_2$, and
 \begin{equation*}
    \Phi(\widetilde  {\mathbf z} ,\widetilde  {\mathbf z}')=(\Phi_1(\widetilde  {\mathbf z }_1 ,\widetilde  {\mathbf z }_1') ,\Phi_2  (\widetilde  {\mathbf z }_2
    ,\widetilde  {\mathbf z
    }_2')),\qquad \widetilde  {\mathbf z}=(\widetilde  {\mathbf z }_1 ,\widetilde  {\mathbf z }_2)\in \mathbb{C}^{n_1}\times  \mathbb{C}^{n_2}=\mathbb{C}^{n },
 \end{equation*}
  with
$
    \Phi_\alpha(\widetilde  {\mathbf z }_\alpha ,\widetilde  {\mathbf z }_\alpha'):=  \langle\widetilde  {\mathbf z }_\alpha ,\widetilde  {\mathbf z
    }_\alpha'\rangle
$, where $\langle\cdot ,\cdot\rangle$ is the standard Hermitian inner product,
   and $\rho =(\rho_1 ,\rho_2 )$  with
   \begin{equation}\label{eq:rho}
   \rho_\alpha(\zeta,\eta)=\mathbf{i}\left(\overline{\widetilde  w'_\alpha }-\widetilde  w_\alpha\right)-2\langle\widetilde  {\mathbf z}_\alpha ,\widetilde
   {\mathbf
   z}_\alpha'\rangle,
\end{equation}for   $\zeta=(\zeta',\zeta'')=(\widetilde  {\mathbf w},\widetilde  {\mathbf z})\in\mathcal{D}$, $\eta=(\eta',\eta'')=(\widetilde  {\mathbf w}',\widetilde  {\mathbf
z}')\in   {\mathcal S}$.
  Then, $ B_\lambda(\widetilde  {\mathbf z }  ,\widetilde  {\mathbf z } ')=4\lambda_1\langle\widetilde  {\mathbf z }_1 ,\widetilde  {\mathbf z
    }_1'\rangle+ 4\lambda_2\langle\widetilde  {\mathbf z }_2 ,\widetilde  {\mathbf z
    }_2'\rangle$, i.e.
\begin{equation*}
   B_\lambda=4
   \left(\begin{matrix} \lambda_1I_{n_1}&0\\
   0&\lambda_2I_{n_2}
  \end{matrix}\right),
\end{equation*}
and so
 $\det B_\lambda=4^{n_1+n_2} \lambda_1^{n_1} \lambda_2^{n_2}  $.
It follows from \eqref{eq:Szego} that
\begin{equation} \label{eq:Szego2}\begin{split}
   S(\zeta, \eta)&= \int_{\mathbb{R}^{2}_+ }e^{-2\pi\sum_{\alpha=1}^{2} \lambda_\alpha\rho_\alpha(\zeta,\eta) }4^{n_1+n_2}\lambda_1^{n_1} \lambda_2^{n_2}  d\lambda
   =\prod_{\alpha=1}^2
 \frac {c_\alpha}{  \rho_\alpha(\zeta,\eta) ^{n_\alpha+1}  },\qquad c_\alpha:=\frac {  n_\alpha! }{ 4(\frac \pi 2)^{n_\alpha+1}},
\end{split}\end{equation}
by applying
\begin{equation*}
   \int_0^{+\infty} e^{-2\pi  s\theta }s^{m}  ds= \frac {m!}{(2\pi \theta)^{m+1}}
\end{equation*}
for $\theta\in \mathbb{C}$ with ${\rm Re}\,\theta>0$.

We need to transform the  Cauchy-Szeg\H o kernel \eqref{eq:Szego2} on $  \mathcal{ U }$  to that on $   \mathscr U $.
\begin{cor} \label{cor:S} The  Cauchy-Szeg\H o kernel of the  Cauchy-Szeg\H o projection $\mathcal P  $ on $   \mathscr U $ is
\begin{equation}\label{eq:flat-Szego1}
   S((\mathbf{t},\mathbf{g}),\mathbf{g}')=\prod_{\alpha=1}^2 S_\alpha\left(t_\alpha,\mathbf{g}_\alpha'{}^{-1}\mathbf{g}_\alpha\right),
\end{equation} for $\mathbf{g}=(\mathbf{g}_1,\mathbf{g}_2), $  $\mathbf{g}'=(\mathbf{g}_1',\mathbf{g}_2') \in \mathscr H_1\times\mathscr H_2 $,
$\mathbf{t}=(t_1,t_2)\in \mathbb{ {R}}_+^2$,
where
\begin{equation}\label{eq:flat-Szego2}
   S_\alpha(t_\alpha, \mathbf{h}_\alpha)=
 \frac {c_\alpha}{ (|{\mathbf z}_\alpha |^2 + t_\alpha- \mathbf{i} s_\alpha ) ^{n_\alpha+1}    },
\end{equation}if we write $\mathbf{h}_\alpha =(s_\alpha ,\mathbf{z}_\alpha )\in \mathscr H_\alpha$, is the  Cauchy-Szeg\H o kernel of the  Cauchy-Szeg\H o projection $\mathcal P_\alpha $ on $   \mathscr U_\alpha $.
   \end{cor}
\begin{proof} Recall that for $f\in H^2(\mathscr U)$,  the function $\widetilde f$ defined by \eqref{eq:tilde-f-f}
 belongs to $H^2( \mathcal{U} )$. Using the    Cauchy-Szeg\H o kernel \eqref{eq:Szego2} on $\mathcal{U}$ and applying the reproducing formula
\eqref{eq:Szego-projection}
to  $\widetilde f$, we get
\begin{equation}\label{eq:flat-Szego3}\begin{split}&
  f(\widetilde  w_1-\mathbf{i}| \widetilde  {\mathbf z}_1|^2,\widetilde w_2-\mathbf{i}|\widetilde  {\mathbf z}_2|^2,\widetilde   {\mathbf z}  )
  = \int_{{\mathbb{R}^2\times \mathbb{C}^{n_1+n_2}}}f(\widetilde  w_1'-\mathbf{i}|\widetilde  {\mathbf z}_1'|^2,\widetilde  w_2'-\mathbf{i}|\widetilde  {\mathbf
  z}_2'|^2,\widetilde
  {\mathbf z} '  ) \prod_{\alpha=1}^2
 \frac {c_\alpha}{ \rho_\alpha(\zeta,\eta) ^{n_\alpha+1}       }d \operatorname{Re}\widetilde  {\mathbf w}' d\widetilde   {\mathbf z}',
\end{split} \end{equation}for $\zeta=(\widetilde  {\mathbf w},\widetilde  {\mathbf z})\in   {\mathcal U}, \eta=(\widetilde  {\mathbf w}',\widetilde  {\mathbf
z}')\in
{\mathcal S}$.
Write
\begin{equation}\label{eq:flat-Szego4}
   \widetilde  w_\alpha =   s_\alpha+\mathbf{i}t_\alpha+\mathbf{i} | {\mathbf z}_\alpha |^2 , \qquad  \widetilde  w_\alpha '=   s_\alpha' + \mathbf{i} | {\mathbf
   z}_\alpha '|^2 ,  \qquad
   \widetilde  {\mathbf z}_\alpha={\mathbf z}_\alpha ,  \qquad  \widetilde  {\mathbf z}_\alpha'={\mathbf z}_\alpha',
\end{equation} with $t_\alpha>0$.
  By definition \eqref{eq:rho}, we have
\begin{equation}\label{eq:flat-Szego5}\begin{split}
   \rho_\alpha(\zeta,\eta)& =\mathbf{i}\left(\overline{\widetilde  w'_\alpha{} }-\widetilde  w_\alpha\right)-2\Phi_\alpha(\widetilde
   {\mathbf z}_\alpha,\widetilde  {\mathbf z}_\alpha')\\&=-\mathbf{i}(s_\alpha-s_\alpha')+\operatorname{Im }\widetilde   w_\alpha +\operatorname{Im } \widetilde
   {w }_\alpha' -2\langle\widetilde  {\mathbf z}_\alpha ,\widetilde  {\mathbf
   z}_\alpha'\rangle\\&=
    -\mathbf{i}(s_\alpha-s_\alpha')+ t_\alpha+ | {\mathbf z}_\alpha |^2+ | {\mathbf z}_\alpha '|^2-2\langle {\mathbf z}_\alpha , {\mathbf z}_\alpha'\rangle\\
   &= - \mathbf{i}\left(s_\alpha-s_\alpha'-2\operatorname{Im}\langle{{\mathbf z}}_\alpha',    {\mathbf z}_\alpha\rangle\right)+ t_\alpha+ | {\mathbf z}_\alpha -
   {\mathbf z}_\alpha'|^2 .
\end{split}\end{equation}
Substituting \eqref{eq:flat-Szego4}-\eqref{eq:flat-Szego5} to \eqref{eq:flat-Szego3} to get the reproducing formula
\begin{equation*}
  f(  \mathbf{w}  ,  {\mathbf z}  )=\int_{{\mathbb{R}^2\times \mathbb{C}^{n_1+n_2}}}f(  \mathbf{s} '   , {\mathbf z} '  ) \prod_{\alpha=1}^2
 \frac {c_\alpha}{ (|{\mathbf z}_\alpha-  {\mathbf z}_\alpha'|^2 + t_\alpha- \mathbf{i}(s_\alpha-s_\alpha'-2\operatorname{Im}\langle{{\mathbf z}}_\alpha',
 {\mathbf
 z}_\alpha\rangle) ^{n_\alpha+1}    }d   {\mathbf s}' d   {\mathbf z}'.
\end{equation*}
  The result follows from the multiplication law \eqref{eq:multiplication} of  the Heisenberg group $\mathscr H_\alpha$.
  \end{proof}

\subsection{Estimates for integral operators  associated to the    Cauchy-Szeg\H o  kernel}

 For an integral operator $T$ with kernel $K(\mathbf{g}_1,\mathbf{g}_2,\mathbf{g}_1' ,\mathbf{g}_2' )$ on $\mathscr H_1\times\mathscr H_2 $, i.e.
 \begin{equation*}
    Tf( \mathbf{g}_1,\mathbf{g}_2)=\int_{\mathscr H_1\times\mathscr H_2}K(\mathbf{g}_1,\mathbf{g}_2,\mathbf{g}_1' ,\mathbf{g}_2' )f(
    \mathbf{g}_1',\mathbf{g}_2')d\mathbf{g}_1'd\mathbf{g}_2',
 \end{equation*}
 and  for fixed $\mathbf{g}_1,\mathbf{g}'_1$, we denote by $K^{(1)}(\mathbf{g}_1,\mathbf{g}'_1)$ the integral operator acting on functions on $\mathscr H_2$ with
 the
   kernel
 \begin{equation*}
    K^{(1)}(\mathbf{g}_1,\mathbf{g}'_1) ( \mathbf{g}_2, \mathbf{g}_2' ):=K(\mathbf{g}_1,\mathbf{g}_2,\mathbf{g}_1' ,\mathbf{g}_2' ) .
 \end{equation*}
 The integral operator $K^{(2)}(\mathbf{g}_2,\mathbf{g}'_2)$  is defined similarly.

The composition of the operator $T$ with   $e^{-\tau_1\triangle_1}$ is the operator $T\circ e^{-\tau_1\triangle_1}$ with kernel
 \begin{equation*}
    K_{\tau_1,0}(\mathbf{g},\mathbf{g}'):=\int_{\mathscr H_1}K(\mathbf{g}_1,\mathbf{g}_2,\mathbf{h}_1
    ,\mathbf{g}_2' )h_{\tau_1}^{(1)}(\mathbf{h}_1^{-1}\mathbf{g}_1' )d\mathbf{h}_1
 \end{equation*}by $h_{\tau_1}^{(1)}(\mathbf{h}_1 ^{-1}  )=h_{\tau_1}^{(1)}(\mathbf{h}_1   )$  for $\mathbf{h}_1\in \mathscr H_1$ (cf. \cite{Gaveau}). Similarly,
$ K_{0,\tau_2}(\mathbf{g},\mathbf{g}')$ and $ K_{\tau_1,\tau_2}(\mathbf{g},\mathbf{g}')$ are integral kernels of $T\circ e^{-\tau_2\triangle_2}$ and $T\circ e^{-\tau_1\triangle_1}\circ e^{-\tau_2\triangle_2}$, respectively.

For fixed $\mathbf{t}\in \mathbb{R}^2_+$, denote
 \begin{equation*}K(\mathbf{g},\mathbf{g}'; \mathbf{t}):= S((\mathbf{t},\mathbf{g}),\mathbf{g}').
     \end{equation*}
  We show that $K(\mathbf{g},\mathbf{g}'; \mathbf{t})$ satisfies the condition of rough bi-parameter  Calder\'on-Zygmund kernels   uniformly for $\mathbf{t} $,   which enables us to prove    that the  Cauchy-Szeg\H o projection maps a $(  2, N)$-atom on $ \mathscr H_1\times \mathscr H_2 $  to an $H^1(\mathscr U)$ function in the next section.  See \cite{DM} \cite{DLY} for rough    Calder\'on-Zygmund operators. The structure of rough bi-parameter  Calder\'on-Zygmund kernels  is especially suitable for estimating the action of corresponding operators on $(  2,N)$-atoms on   product spaces.

 \begin{prop}\label{prop:K-estimate}For any $\gamma_1,\tau_1,\gamma_2,\tau_2>0$, there exists an absolute constant $C>0$ such that
  \begin{equation}\label{eq:K-estimate} \begin{split}  &
     \int_{\|\mathbf{g}_1'{}^{-1}\mathbf{g}_1\|_1>\gamma_1\tau_1}\left\|  K^{(1)}(\mathbf{g}_1,\mathbf{g}'_1;
     \mathbf{t})-K_{\tau_1^2,0}^{(1)}(\mathbf{g}_1,\mathbf{g}'_1; \mathbf{t})  \right\|_{L^2(\mathscr H_2)\rightarrow L^2(\mathscr H_2)}d\mathbf{g}_1\leq   C
     \gamma_1^{-2},\\ &
     \int_{\|\mathbf{g}_2'{}^{-1}\mathbf{g}_2\|_2>\gamma_2\tau_2}\left\|  K^{(2)}(\mathbf{g}_2,\mathbf{g}'_2;
     \mathbf{t})-K_{0,\tau_2^2}^{(2)}(\mathbf{g}_2,\mathbf{g}'_2; \mathbf{t})  \right\|_{L^2(\mathscr H_1)\rightarrow L^2(\mathscr H_1)}d\mathbf{g}_2\leq   C
     \gamma_2^{-2},\\
     &\int_{\substack {\|\mathbf{g}_1'{}^{-1}\mathbf{g}_1\|_1>\gamma_1\tau_1\\
     \|\mathbf{g}_2'{}^{-1}\mathbf{g}_2\|_2>\gamma_2\tau_2}}\left | K(\mathbf{g},\mathbf{g}'; \mathbf{t})-
     K_{\tau_1^2,0}(\mathbf{g},\mathbf{g}'; \mathbf{t})-K_{0,\tau_2^2}(\mathbf{g},\mathbf{g}'; \mathbf{t})+K_{\tau_1^2,\tau_2^2}(\mathbf{g},\mathbf{g}';
     \mathbf{t})\right|d\mathbf{g}  \leq     C \gamma_1^{-2} \gamma_2^{-2}   .
 \end{split} \end{equation}
\end{prop}
 \begin{proof} Note that the Cauchy-Szeg\H o  kernel $S_\alpha\left(t_\alpha , \mathbf{g}_\alpha\right) $ in \eqref{eq:flat-Szego2} belongs to    holomorphic Hardy space $  H^2(\mathscr
 U_\alpha )$. Consequently,  $S_\alpha\left(t_\alpha+\cdot , \cdot\right) $ for fixed $t_\alpha>0$ has  smooth boundary value $S_\alpha\left( t_\alpha ,
 \cdot\right) $, and  also belongs to   holomorphic Hardy space $  H^2(\mathscr U_\alpha )$ by definition. Thus by Proposition \ref{prop:heat-kernel-H1} for $\mathscr U_\alpha
 $, we have
 \begin{equation}\label{eq:S-t+t}
    S_\alpha\left(t_\alpha+ s_\alpha, \mathbf{g}_\alpha\right)=\int_{\mathscr H_\alpha }h_{ {s_\alpha } }^{(\alpha)}(\mathbf{h}_\alpha ^{-1}
    \mathbf{g}_\alpha)S_\alpha\left( t_\alpha , \mathbf{h}_\alpha \right) d \mathbf{h}_\alpha
 \end{equation}
Noting that $S_\alpha\left(t_\alpha ,\mathbf{h} _\alpha^{-1} \right)=\overline{S_\alpha\left(t_\alpha ,\mathbf{h} _\alpha  \right)}$ by its expression in  Corollary \ref{cor:S},  we
get
  \begin{equation}\label{eq:K*h}\begin{split}
     K_{\tau_1^2,0}(\mathbf{g},\mathbf{g}'; \mathbf{t})&=\int_{\mathscr H_1 }h_{\tau_1^2}^{(1)}(\mathbf{h}_1  ^{-1}\mathbf{g}_1' )S_1\left(t_1
     ,\mathbf{h}_1^{-1}\mathbf{g} _1\right)d\mathbf{h}_1 \cdot S_2\left(t_2,\mathbf{g}_2'{}^{-1}\mathbf{g}_2\right)
     \\&=\overline{\int_{\mathscr H_1 }h_{\tau_1^2}^{(1)}(\mathbf{h}_1 ^{-1}\mathbf{g}_1' )S_1\left(t_1 ,\mathbf{g}
     _1^{-1}\mathbf{h}_1\right)d\mathbf{h}_1 }\cdot S_2\left(t_2,\mathbf{g}_2'{}^{-1}\mathbf{g}_2\right)\\&=\overline{ S_1\left(t_1+\tau_1^2
     ,\mathbf{g}_1^{-1}\mathbf{g}_1'\right)  }  S_2\left(t_2,\mathbf{g}_2'{}^{-1}\mathbf{g}_2\right)
     \\&=
     S_1\left(t_1+\tau_1^2,\mathbf{g}_1'{}^{-1}\mathbf{g}_1\right)S_2\left(t_2,\mathbf{g}_2'{}^{-1}\mathbf{g}_2\right) .
  \end{split} \end{equation}by   using \eqref{eq:S-t+t} and reality of $h ^{(1)}$.
Then
   \begin{equation}\label{eq:K-K}\begin{split}
    K(\mathbf{g},\mathbf{g}'; \mathbf{t})-
     K_{\tau_1^2,0}(\mathbf{g},\mathbf{g}'; \mathbf{t})&=\left[S_1\left(t_1
     ,\mathbf{g}_1'{}^{-1}\mathbf{g}_1\right)-S_1\left(t_1+\tau_1^2,\mathbf{g}_1'{}^{-1}\mathbf{g}_1\right)\right]S_2\left(t_2,\mathbf{g}_2'{}^{-1}\mathbf{g}_2\right).
 \end{split} \end{equation}

Now if  we write
$
   \mathbf{g}_\alpha'{}^{-1}\mathbf{g}_\alpha =(s_\alpha ,\mathbf{z}_\alpha )\in \mathscr H_\alpha,
$
then
  \begin{equation}\label{eq:S-S}\begin{split}
   & \left|S_\alpha\left(t_\alpha
   ,\mathbf{g_\alpha'}{}^{-1}\mathbf{g}_\alpha\right)-S_\alpha\left(t_\alpha+\tau_\alpha^2,\mathbf{g}_\alpha'{}^{-1}\mathbf{g}_\alpha\right)\right|\\=&c_\alpha
 \left|\frac {(|{\mathbf z}_\alpha |^2 + t_\alpha+\tau_\alpha^2- \mathbf{i} s_\alpha ) ^{n_\alpha+1}- (|{\mathbf z}_\alpha |^2 + t_\alpha- \mathbf{i} s_\alpha )
 ^{n_\alpha+1}  }{ (|{\mathbf z}_\alpha |^2 + t_\alpha- \mathbf{i} s_\alpha ) ^{n_\alpha+1} (|{\mathbf z}_\alpha |^2 + t_\alpha+\tau_\alpha^2- \mathbf{i} s_\alpha
 ) ^{n_\alpha+1}  }\right|\\=&c_\alpha
 \left|\frac  {\tau_\alpha^2 \sum_{a=0}^{n_\alpha} (|{\mathbf z}_\alpha |^2 + t_\alpha+\tau_\alpha^2- \mathbf{i} s_\alpha ) ^{a} (|{\mathbf z}_\alpha |^2 +
 t_\alpha- \mathbf{i} s_\alpha ) ^{n_\alpha-a} }{(|{\mathbf z}_\alpha |^2 + t_\alpha- \mathbf{i} s_\alpha ) ^{n_\alpha+1}  (|{\mathbf z}_\alpha |^2 +
 t_\alpha+\tau_\alpha^2- \mathbf{i} s_\alpha ) ^{n_\alpha+1} }\right|
 \\ \leq &c_\alpha'
  \frac {\tau_\alpha^2   }{ ||{\mathbf z}_\alpha |^2  - \mathbf{i} s_\alpha | ^{n_\alpha+2}  } =
 c_\alpha'
 \frac { \tau_\alpha^2   }{  \| \mathbf{g}_\alpha'{}^{-1}\mathbf{g}_\alpha  \|_\alpha^{Q_\alpha+2 } }
 \end{split} \end{equation}
 by Corollary \ref{cor:S}, where $c_\alpha':=c_\alpha (n_\alpha+1)$.
Therefore, for $f \in L^2(\mathscr H_2)$,
   \begin{equation*}\begin{split}
    & \left\| \left[K^{(1)}(\mathbf{g}_1,\right.\right. \mathbf{g}'_1; \mathbf{t})-\left.\left.K_{\tau_1^2,0}^{(1)}(\mathbf{g}_1,\mathbf{g}'_1; \mathbf{t})\right]f \right\|_{L^2(\mathscr H_2)  }\\
=  &\left(\int_{\mathscr H_2}\left|\int_{\mathscr H_2}[K(\mathbf{g},\mathbf{g}'; \mathbf{t})-
     K_{\tau_1^2,0}(\mathbf{g},\mathbf{g}'; \mathbf{t})]f( \mathbf{g}_2')d\mathbf{g}_2'\right|^2d\mathbf{g}_2\right)^{\frac 12}\\
=&\left|S_1\left(t_1 ,\mathbf{g}_1'{}^{-1}\mathbf{g}_1\right)-S_1\left(t_1+\tau_1^2,\mathbf{g}_1'{}^{-1}\mathbf{g}_1\right)\right| \left(\int_{\mathscr
H_2}\left|\int_{\mathscr H_2}S_2\left(t_2,\mathbf{g}_2'{}^{-1}\mathbf{g}_2\right)f( \mathbf{g}_2')d\mathbf{g}_2'\right|^2d\mathbf{g}_2\right)^{\frac 12}\\
= &\left|S_1\left(t_1 ,\mathbf{g}_1'{}^{-1}\mathbf{g}_1\right)-S_1\left(t_1+\tau_1^2,\mathbf{g}_1'{}^{-1}\mathbf{g}_1\right)\right|
    \|(\mathcal{P}_2f)(t_2,\cdot )\|_{L^2(\mathscr H_2) }\\
 \leq &   c_1'
 \frac { \tau_1^2   }{  \| \mathbf{g}_1'{}^{-1}\mathbf{g}_1  \|_1^{Q_1+2 }  } \|f\|_{L^2(\mathscr H_2) }
 \end{split} \end{equation*}
 by  \eqref{eq:K-K}-\eqref{eq:S-S}, where for any $t_2>0$,  $f\mapsto (\mathcal{P}_2f)(t_2,\cdot )$ is bounded on $  L^2(\mathscr H_2)$ for the  Cauchy-Szeg\H o projection
 $\mathcal{P}_2$ on $\mathscr U_2$ with the norm $ \leq 1$.
 Therefore,
 \begin{equation*}\begin{split}
    & \int_{\|\mathbf{g}_1'{}^{-1}\mathbf{g}_1\|_1>\gamma_1\tau_1}\left\|  K^{(1)}(\mathbf{g}_1,\mathbf{g}'_1;
     \mathbf{t})-K_{\tau_1^2,0}^{(1)}(\mathbf{g}_1,\mathbf{g}'_1; \mathbf{t})  \right\|_ {L^2(\mathscr H_2)\rightarrow L^2(\mathscr H_2)} d\mathbf{g}_1 \\
\leq&
     \int_{\| \mathbf{g}_1'{}^{-1}\mathbf{g}_1 \|_1>\gamma_1\tau_1}
 \frac { c_1'\tau_1^2   }{ \| \mathbf{g}_1'{}^{-1}\mathbf{g}_1  \|_1^{Q_1 +2}  }d\mathbf{g}_1
=   \frac { c_1' }{ \gamma_1^2  }  \int_{\| \mathbf{h}_1   \|_1>1}
 \frac {1 }{  \| \mathbf{h}_1   \|_1^{Q_1 +2}  }d\mathbf{h}_1
  = C \gamma_1^{-2} ,
 \end{split} \end{equation*}by rescaling and using  the invariance of the measure $d\mathbf{g}_1$. The first estimate in \eqref{eq:K-estimate} is proved. It is similar
 to show the second estimate in \eqref{eq:K-estimate}.

As in \eqref{eq:K*h}, we have
 \begin{equation}\label{eq:4K0}\begin{split}
   K_{0,\tau_2^2}(\mathbf{g},\mathbf{g}'; \mathbf{t})&=S_1\left(t_1 ,\mathbf{g}_1'{}^{-1}\mathbf{g}_1\right)
     S_2\left(t_2+\tau_2^2,\mathbf{g}_2'{}^{-1}\mathbf{g}_2\right) , \\K_{\tau_1^2,\tau_2^2}(\mathbf{g},\mathbf{g}'; \mathbf{t})
     &=S_1\left(t_1+\tau_1^2,\mathbf{g}_1'{}^{-1}\mathbf{g}_1\right)
     S_2\left(t_2+\tau_2^2,\mathbf{g}_2'{}^{-1}\mathbf{g}_2\right).
 \end{split} \end{equation}
Thus,
 \begin{equation*} \begin{split}
   &\left | K(\mathbf{g},\mathbf{g}'; \mathbf{t})-
     K_{\tau_1^2,0}(\mathbf{g},\mathbf{g}'; \mathbf{t})-K_{0,\tau_2^2}(\mathbf{g},\mathbf{g}'; \mathbf{t})+K_{\tau_1^2,\tau_2^2}(\mathbf{g},\mathbf{g}';
     \mathbf{t})\right|\\
     =&\left|\left[S_1\left(t_1 ,\mathbf{g}_1'{}^{-1}\mathbf{g}_1\right)-S_1\left(t_1+\tau_1^2,\mathbf{g}_1'{}^{-1}\mathbf{g}_1\right)\right]S_2
     \left(t_2,\mathbf{g}_2'{}^{-1}\mathbf{g}_2\right)\right.\\
     &\quad-\left.\left[S_1\left(t_1 ,\mathbf{g}_1'{}^{-1}\mathbf{g}_1\right)-S_1\left(t_1+\tau_1^2,\mathbf{g}_1'{}^{-1}\mathbf{g}_1\right)\right]
     S_2\left(t_2+\tau_2^2,\mathbf{g}_2'{}^{-1}\mathbf{g}_2\right)\right|\\
     =&\left| S_1\left(t_1 ,\mathbf{g}_1'{}^{-1}\mathbf{g}_1\right)-S_1\left(t_1+\tau_1^2,\mathbf{g}_1'{}^{-1}\mathbf{g}_1\right)\right|\left|S_2
     \left(t_2,\mathbf{g}_2'{}^{-1}\mathbf{g}_2\right)-
     S_2\left(t_2+\tau_2^2,\mathbf{g}_2'{}^{-1}\mathbf{g}_2\right) \right|\\\leq &\prod_{\alpha=1}^2
   \frac { c_\alpha'   \tau_\alpha^2   }{  \| \mathbf{g}_\alpha'{}^{-1}\mathbf{g}_\alpha  \|_\alpha^{Q_\alpha+2} }
 \end{split} \end{equation*}
by estimate \eqref{eq:S-S}. Then,
\begin{equation*} \begin{split}
   &\int_{\substack {\|\mathbf{g}_1'{}^{-1}\mathbf{g}_1\|_1>\gamma_1\tau_1\\ \|\mathbf{g}_2'{}^{-1}\mathbf{g}_2\|_2>\gamma_2\tau_2}}\left | K(\mathbf{g},\mathbf{g}';
   \mathbf{t})-
     K_{\tau_1^2,0}(\mathbf{g},\mathbf{g}'; \mathbf{t})-K_{0,\tau_2^2}(\mathbf{g},\mathbf{g}'; \mathbf{t})+K_{\tau_1^2,\tau_2^2}(\mathbf{g},\mathbf{g}';
     \mathbf{t})\right|d\mathbf{g}_1d\mathbf{g}_2\\
     \leq &
  c_1' c_2'
\int_{\substack {\|\mathbf{g}_1'{}^{-1}\mathbf{g}_1\|_1>\gamma_1\tau_1\\ \|\mathbf{g}_2'{}^{-1}\mathbf{g}_2\|_2>\gamma_2\tau_2}} \frac { \tau_1^2 \tau_2^2  }{  \|
\mathbf{g}_1'{}^{-1}\mathbf{g}_1 \|_1^{Q_1+2}  \| \mathbf{g}_2'{}^{-1}\mathbf{g}_2  \|_2^{Q_2+2}}d\mathbf{g}_1d\mathbf{g}_2\\
     \leq &
 \frac {  c_1' c_2'  }{ \gamma_1^2 \gamma_2^2  }
\int_{\substack {\| \mathbf{h}_1\|_1>1\\ \| \mathbf{h}_2\|_2>1}} \frac { 1 }{  \| \mathbf{h}_1 \|_1^{Q_1+2}  \|  \mathbf{h}_2  \|_2^{Q_2+2}}d\mathbf{h}_1d\mathbf{h}_2
=  \frac { C }{ \gamma_1^2 \gamma_2^2  }
 \end{split} \end{equation*}by rescaling and using  the invariance of the measure $d\mathbf{g}_1d\mathbf{g}_2$ again.
\end{proof}

\section{ The boundedness of the   Cauchy-Szeg\H o  projection  from $H^1(\mathscr H_1\times \mathscr H_2)$  to $H^1(\mathscr U)$   }
\subsection{ Journ\'e's covering Lemma   in the setting of spaces
of homogeneous type}
We need  an analogue on
spaces of homogeneous type of the grid of Euclidean dyadic cubes  by Christ.

\begin{lem} \label{lem:Christ} \cite{ch} Let $(X, d,\mu)$ be a space of homogeneous type. Then there exist a collection $\{I_\alpha^ k \subset X ; \alpha \in \mathcal{ I}_k ,k \in \mathbb{
Z}\}$ of open subsets of $X$, where $\mathcal{ I}_k$ is some countable index set, and
constants $C_1 , C_2 > 0$, such that
\\
(i) $\mu (X\setminus \bigcup_{\alpha} I_\alpha^ k) = 0$ for each fixed $k$, and $I_\alpha^ k\bigcap I_\beta^ k = \emptyset$ if $\alpha\neq\beta $;
\\
(ii) for all $\alpha, \beta, k,l$ with $l \geq k$, either $I_\alpha^ k\bigcap I_\beta^ l = \emptyset$ or $I_\alpha^ k \supset I_\beta^ l  $;
\\
(iii) for each $(k, \alpha)$ and each $l < k$ there is a unique $\beta$ such that $I_\alpha^ k \subset I_\beta^ l  $;
\\
(iv) $l(I_\alpha^ k ):=\operatorname{diam}(I_\alpha^ k)\leq C_12^{-k}$; and
\\
(v) each $I_\alpha^ k    $
  contains some ball $B(z_\alpha^ k,C_22^{-k} ) $, where $z_\alpha^ k\in X$.\end{lem}

We can choose the absolute constant $\overline{C}$ in the definition of a $(  2,N)$-atom sufficiently large  so that we can take $C_1=\overline{C}$  and $C_2= \overline{C}^{-1}$ in (iv)-(v).
The point $z_\alpha^ k $ is called the {\it center} of the set $I_\alpha^ k    $. We also call $I_\alpha^ k    $   a {\it dyadic
cube} with diameter roughly $\overline{C}2^{-k}   $, centered at $  z_\alpha^ k   $.   We refer to the set $\lambda I_\alpha^ k    $ as the cube with the same center as
$I_\alpha^ k    $
and diameter $\lambda\operatorname{diam}(I_\alpha^ k) $.
 Let $\{I_\alpha^ k ;\alpha\in  \mathcal{ I}_k ,k \in \mathbb{
Z}  \} $ and $\{ J_\beta^ l; \beta\in \mathcal{J}_l ,l \in \mathbb{
Z} \}  $
  be dyadic cubes on the   Heisenberg groups $ \mathscr H_1$ and $ \mathscr H_2$, respectively, given by  Lemma \ref{lem:Christ}.   The open set $I_\alpha^ k \times
  J_\beta^ l   $ for $  \alpha\in  \mathcal{ I}_k , \beta\in  \mathcal{ J}_l$ ($k ,l \in \mathbb{ Z}$) is called
a {\it dyadic rectangle} in $ \mathscr H_1\times\mathscr H_2$.

For an open set $\Omega$ in $\mathscr H_1\times \mathscr H_2$ with finite measure and
  each rectangle $R = I \times J$, let $ I ^*$ be the largest dyadic cube in $\mathscr H_1$ containing $I$
such that $I ^*\times  J \subset \tilde{\Omega} $, where
$\tilde{\Omega}:= \{\mathbf{g} \in \mathscr H_1\times \mathscr H_2: M_S(\chi_{\Omega}
)(\mathbf{g}) > 1/2\}$ and $M_S$ denotes
the strong maximal function. Next, let $J^*$ be the largest dyadic cube in $\mathscr H_2$ containing $ J $ such
that $I ^*\times  J^* \subset \tilde{\tilde{\Omega} }$, where $\widetilde{\widetilde{\Omega}}:= \{\mathbf{g} \in \mathscr H_1\times \mathscr H_2;
M_S({\chi_{\widetilde\Omega}} )(\mathbf{g}) > 1/2\}$.
Now let
\begin{equation*}
   R^* =
{\breve{C}}I ^*\times  {\breve{C}}J^*,\qquad \text{ where }\quad {\breve{C}}=2\overline{C}^3.
\end{equation*}
  An application of the strong maximal function
theorem shows that
\begin{equation*}
  \left|\bigcup_{R\subset \Omega} R^*\right|\leq C\left|\tilde{\tilde{\Omega} }\right|\leq C|\tilde{\Omega} |\leq C| \Omega  |.
\end{equation*}

\begin{lem}\label{lem:Journe} \cite{HLL} Denote by $ m_\alpha(\Omega)$ the family of dyadic
rectangles $ R\subset\Omega$ which are maximal in the $\mathbf{g}_\alpha$-direction, for $\alpha = 1, 2$. Let $\Omega$ be  an open subset of $ \mathscr H_1\times \mathscr
H_2$ with finite measure and $\kappa>0$. Then
\begin{equation}\label{eq:Journe1}\begin{split}
& \sum_{R= I \times J\in m_1(\Omega)} |R| \left(\frac {\ell (J)}{\ell (J^*)}\right)^\kappa \leq    C     |\Omega|  ,
\\&
 \sum_{R= I \times J\in m_2(\Omega)} |R|  \left(\frac {\ell (I)}{\ell (I^*)}\right)^\kappa \leq    C     |\Omega|  ,
\end{split} \end{equation}for some constant $C$ independent of $\Omega$.
\end{lem}
\subsection{
  The  action of the   Cauchy-Szeg\H o
projection on  $(  2, N)$-atoms}
\begin{thm}\label{prop0}     For any   $(  2, N)$-atom $a$ on the  group $\mathscr H_1\times \mathscr H_2$,
$\mathcal P (a)$ belongs to $H^1(\mathscr U )$ with
\begin{equation*}
   \|\mathcal P (a)\|_{H^1(\mathscr U )}\leq C_{ Q_1,Q_2,N}
\end{equation*}
  for some constant $C_{ Q_1,Q_2,N}$  depending only on $ Q_1,Q_2, N$.
\end{thm}
\begin{proof}
Let $a= \sum_{R \in m (\Omega)} a_{R}$.
Note that for fixed $\mathbf{t}\in \mathbb{R}_+^2$,
\begin{equation*}
   \|\mathcal{ P}a (\mathbf{t},\cdot)\|_{L^1(\mathscr H_1\times \mathscr H_2)}= \| \mathcal{P}a (\mathbf{t},\cdot)\|_{L^1(\cup R^*)}+ \| \mathcal{P}a
   (t,\cdot)\|_{L^1((\cup R^*)^c)}.
\end{equation*}
Since $\mathcal{P}$ is bounded from $L^2(\mathscr H_1\times \mathscr H_2)$  to $ H^2(\mathscr U )$,  by H\"older's inequality, we have
\begin{equation}\begin{split}
  \|\mathcal{ P}a (\mathbf{t},\cdot)\|_{L^1(\cup R^*)}&=  \int_{\cup R^*}|\mathcal{ P}a (\mathbf{t},\mathbf{g})|d\mathbf{g}
  \leq C  |\cup R^*|^{\frac 12}\left (\int_{\cup R^*}|\mathcal{ P}a (\mathbf{t},\mathbf{g})|^2d\mathbf{g}\right) ^{\frac 12}
\\ & \leq C |\Omega| ^{\frac 12}\|a \|_{L^2(\mathscr H_1\times \mathscr H_2)}\leq C |\Omega| ^{\frac 12}  |\Omega| ^{-\frac 12}=C  .
\end{split} \end{equation}
Thus, to prove the theorem,  it is sufficient   to verify the   uniform  boundedness of $ \| \mathcal{P}a
   (t,\cdot)\|_{L^1((\cup R^*)^c)}$. Write
\begin{equation}\label{eq:I1-I2}\begin{split}
  \| \mathcal{P}a (\mathbf{t},\cdot)\|_{L^1((\cup R^*)^c)}&=  \int_{ (\cup R^*)^c }| \mathcal{P}a (\mathbf{t},\mathbf{g})|d\mathbf{g}
 \leq \sum_{R\in m(\Omega)}  \int_{ (  R^*)^c }| \mathcal{P}a_R (\mathbf{t},\mathbf{g})|d\mathbf{g}  \\
&\leq \sum_{R\in m(\Omega)}  \int_{ ({\breve{C}}I^*)^c \times \mathscr H_2}| \mathcal{P}a_R (\mathbf{t},\mathbf{g})|d\mathbf{g} +  \sum_{R\in m(\Omega)}  \int_{ \mathscr
H_1\times
({\breve{C}}J^*)^c}|\mathcal{ P}a_R (\mathbf{t},\mathbf{g})|d\mathbf{g}\\
&=: \sum_{R \in m (\Omega)} I^{R}_{1}+\sum_{R \in m (\Omega)} I^{R}_{ 2 }.
\end{split} \end{equation}

For $R= I \times J$, denote  $\ell_1=l(I)$, $\ell_2=l(J)$, and
\begin{equation*}
   a_{R,1}:=\left(\triangle_1^N \otimes \triangle_2^{N -1}\right)b_R,\qquad  a_{R,2}:=\left(\triangle_1^{N-1} \otimes \triangle_2^N \right)b_R.
\end{equation*}
Then, we have
\begin{equation*}
   a_{R }=\left(\triangle_1\otimes {\rm Id}_2\right)   a_{R,2}=\left({\rm Id}_1\otimes\triangle_2\right)a_{R,1},
\end{equation*}where ${\rm Id}_\alpha$ is the identity operator on $L^2(\mathscr H_\alpha)$, $\alpha=1,2$.

Now for the term $ I^{R}_{ 1 }$, we can write
\begin{equation*}\begin{split}
    I^{R}_{1}&= \int_{ ({\breve{C}}I^*)^c \times {\breve{C}}J^*}|\mathcal{ P}a_R (\mathbf{t},\mathbf{g})|d\mathbf{g}+ \int_{ ({\breve{C}}I^*)^c \times ({\breve{C}}J^*)^c}| \mathcal{P}a_R
    (\mathbf{t},\mathbf{g})|d\mathbf{g}
    =:I^{R}_{11}+I^{R}_{12}.
\end{split}\end{equation*}
We decompose the identity operator on $L^2(\mathscr H_1)$ as follows:
\begin{equation}\begin{split}\label{id1}
     {\rm Id}_1&=\left(\frac 2{\ell_1^2}\int_{0}^{\ell_1} sds\right){\rm Id}_1\\
     &= \frac 2{\ell_1^2}\int_{0}^{\ell_1}s ({\rm Id}_1-e^{-s^2\triangle_1 })ds+\frac 2{\ell_1^2}\int_{0}^{\ell_1}s  e^{-s^2\triangle_1 } ds
     =: {\rm Id}_{\ell_1,1}+{\rm Id}_{\ell_1,2}.
\end{split}\end{equation}
Then for $I^{R}_{11}$, we have
\begin{equation*}\begin{split}
  I^{R}_{11} &= \int_{ ({\breve{C}}I^*)^c \times {\breve{C}}J^*}| \mathcal{P}\circ {\rm Id}_1 (a_R)  (\mathbf{t},\mathbf{g})|d\mathbf{g} \\
    &\leq \int_{ ({\breve{C}}I^*)^c \times {\breve{C}}J^*}| \mathcal{P}\circ {\rm Id}_{\ell_1,1} (a_R)  (\mathbf{t},\mathbf{g})|d\mathbf{g}+\int_{ ({\breve{C}}I^*)^c \times
    {\breve{C}}J^*}|\mathcal{ P}\circ {\rm
    Id}_{\ell_1,2} (a_R)  (\mathbf{t},\mathbf{g})|d\mathbf{g}\\
    &=:I^{R}_{111}+I^{R}_{112}.
\end{split}\end{equation*}
For $I^{R}_{111}$, we have
\begin{equation*}\begin{split}
  I^{R}_{111} &\leq \frac 2{\ell_1^2}\int_{0}^{\ell_1} \int_{ ({\breve{C}}I^*)^c \times {\breve{C}}J^*}\left| \mathcal{P} \circ ({\rm Id}_1-e^{-s^2\triangle_1 })(a_R)
  (\mathbf{t},\mathbf{g})\right|d\mathbf{g}sds.
\end{split}\end{equation*}

Now we will use the following notation: for a given function $f$ on $\mathscr H_1\times \mathscr H_2$ and fixed $\mathbf{g}_1\in \mathscr H_1$, let
$f_{\mathbf{g}_1}$ be a function on  $  \mathscr H_2$ given by
$
   f_{\mathbf{g}_1}(\mathbf{g}_2):= f (\mathbf{g}_1,\mathbf{g}_2).
$
Then,
 \begin{equation*}\begin{split}
      \mathcal{ P }\circ ({\rm Id}_1-e^{-s^2\triangle_1 })(a_R)  (\mathbf{t},\mathbf{g}_1,\mathbf{g}_2)
 &= \left[\left(\mathcal{P}  -\mathcal{P}\circ e^{-s^2\triangle_1 }\right)(a_R)\right]_{\mathbf{g}_1}(\mathbf{t} ,\mathbf{g}_2)\\
 &=\int_{{\overline{C}} I}\left [K^{(1)}(\mathbf{g}_1,\mathbf{g}'_1; \mathbf{t})-K_{ s^2,0}^{(1)}(\mathbf{g}_1,\mathbf{g}'_1; \mathbf{t})  \right](a_R) _{\mathbf{g}_1'}  (\mathbf{g}_2)d\mathbf{g}_1'
\end{split}   \end{equation*}
by definition of operators $K^{(1)}(\mathbf{g}_1,\mathbf{g}'_1; \mathbf{t})$ and $K_{s^2,0}^{(1)}(\mathbf{g}_1,\mathbf{g}'_1; \mathbf{t})$ for fixed $\mathbf{t}\in
\mathbb{R}_+^2$. Therefore, by H\"older's and Minkowski's inequalities, we get
  \begin{equation*}\begin{split}
 I^{R}_{111} &
\leq\frac 2{\ell_1^2}\int_{0}^{\ell_1}sds \int_{ ({\breve{C}}I^*)^c \times {\breve{C}}J^*}\left|\left[\left(\mathcal{P}-\mathcal{P}\circ e^{-s^2\triangle_1 }\right)(a_R)\right]_{\mathbf{g}_1}  (\mathbf{t } ,\mathbf{g}_2)\right|    d\mathbf{g}_1  d\mathbf{g}_2\\
  & \leq \frac {2 }  {\ell_1^2}\int_{0}^{\ell_1}sds \int_{ ({\breve{C}}I^*)^c }d\mathbf{g}_1\left(\int_{   {\breve{C}}J^*} \left| \left[\left(\mathcal{P}
  -\mathcal{P}\circ e^{-s^2\triangle_1 }\right)(a_R)\right]_{\mathbf{g}_1}  (\mathbf{t } ,\mathbf{g}_2)\right|^2 d\mathbf{g}_2\right)^{\frac 12}| {\breve{C}}J^*| ^{\frac 12} \\
  &\leq  \frac {C| J | ^{\frac 12}} {\ell_1^2}\int_{0}^{\ell_1} sds  \int_{ ({\breve{C} } I^*)^c }d\mathbf{g}_1 \int_{{\overline{C}} I}\left(\int_{   {\breve{C} } J^*}\left|\left[K^{(1)}(\mathbf{g}_1,\mathbf{g}'_1; \mathbf{t})-K_{s^2,0}^{(1)}(\mathbf{g}_1,\mathbf{g}'_1; \mathbf{t})  \right]
  (a_R) _{\mathbf{g}_1'}  ( \mathbf{g}_2 )  \right|^2d\mathbf{g} _2\right)^{\frac 12} d\mathbf{g}_1' \\
  &\leq  \frac {C| J | ^{\frac 12}} {\ell_1^2}\int_{0}^{\ell_1}sds  \int_{ ({\breve{C}}I^*)^c } d\mathbf{g}_1\int_{{\overline{C}} I} \left \|K^{(1)}(\mathbf{g}_1,\mathbf{g}'_1; \mathbf{t})-K_{s^2,0}^{(1)}(\mathbf{g}_1,\mathbf{g}'_1; \mathbf{t})  \right\|_{L^2 \rightarrow  L^2 }\|(a_R) _{\mathbf{g}_1'} \|_{ L^2(\mathscr H_2)   }d\mathbf{g}_1'.
\end{split}   \end{equation*}

Let $I =  I_\alpha^k$. Then,  $I^*=\lambda I_\alpha^k$ for some $\lambda\geq 1$.
By the definition of dyadic
cubes, we have ${\breve{C} } I^*\supset B(z_\alpha^ k,2\overline{C}^2\lambda 2^{-k} )$ and so $\mathbf{g}_1\in ({\breve{C} } I^*)^c \subset B(z_\alpha^ k,2\overline{C}^2\lambda 2^{-k} )^c$, while
$\mathbf{g}' _1\in \overline{C}I_\alpha^k \subset B(z_\alpha^ k, \overline{C}^2  2^{-k} ) $. Therefore,
\begin{equation*}
   \| {\mathbf g'_1} ^{-1}\mathbf{g}_1\|_1> \overline{C}  l(I^*)= \overline{C} \frac {l(I^*)}{l(I )}\ell_1\geq \overline{C} \frac {l(I^*)}{l(I )}s.
\end{equation*}
Denote $\gamma_1(R):= \overline{C}\frac {l(I^*)}{l(I )}$. Then we find that
\begin{equation*}\begin{split}
I^{R}_{111} & \leq \frac {C | J | ^{\frac 12}}{\ell_1^2}\int_{0}^{\ell_1}sds \int_{{\overline{C}} I} d\mathbf{g}_1' \int_{
\| {\mathbf g'_1} ^{-1}\mathbf{g}_1\|>\gamma_1(R) s } \left \|K^{(1)}(\mathbf{g}_1,\mathbf{g}'_1; \mathbf{t})-K_{s^2,0}^{(1)}(\mathbf{g}_1,\mathbf{g}'_1; \mathbf{t})
\right\|_{L^2 \rightarrow  L^2 }
 \|(a_R) _{\mathbf{g}_1'} \|_{ L^2   } d\mathbf{g}_1\\&
\leq \frac{ C| J | ^{\frac 12}}{\ell_1^2}\int_{0}^{\ell_1} sds\int_{{\overline{C}} I}    \gamma_1(R)^{-2} \|(a_R) _{\mathbf{g}_1'}  \|_{L^2(\mathscr H_2)
}d\mathbf{g}_1'\\&
\leq  C \gamma_1(R)^{-2}  | J | ^{\frac 12}\left(\int_{{\overline{C}} I}     \|(a_R) _{\mathbf{g}_1'}  \|^2_{L^2(\mathscr H_2)  }d\mathbf{g}_1'\right)^{\frac 12} | I | ^{\frac
12}
\\&
\leq  C \gamma_1(R)^{-2}  | R | ^{\frac 12}     \| a_R   \| _{L^2(\mathscr H_1 \times\mathscr H_2)}  ,
\end{split}   \end{equation*}
by using Proposition \ref{prop:K-estimate}. Apply Journ\'e's covering Lemma
 \ref{lem:Journe}   to get
 \begin{equation*}\begin{split}
   \sum_{R \in m (\Omega)} I^{R}_{111}&\leq C \left(\sum_{R \in m_2(\Omega)}\gamma_1(R)^{-4}  | R | \right)^{\frac 12}   \left( \sum_{R \in m (\Omega)}
  \left \| (\triangle_1^N \otimes \triangle_2^N  )b_R  \right \|^2 _{L^2(\mathscr H_1 \times\mathscr H_2)}\right)^{\frac 12}\\
   &\leq C |\Omega|^{\frac 12}|\Omega|^{-\frac 12}=C,
 \end{split}\end{equation*}by the condition of a $(  2, N)$-atom for $\sigma_1=N,\sigma_2=N $ in \eqref{eq:atom}.

To estimate $I^{R}_{112}$, note that
\begin{equation}\label{eq:Id-l2}\begin{split}
    {\rm Id}_{\ell_1,2}(a_R) & =  \left( \frac 2{\ell_1^2}\int_{0}^{\ell_1} s e^{-s^2\triangle_1 } ds\right)\left(\triangle_1\otimes {\rm Id}_2\right)
    a_{R,2}
     =  \left( \frac 2{\ell_1^2}\int_{0}^{\ell_1} s\triangle_1 e^{-s^2\triangle_1 } ds\right)    a_{R,2}\\
      &=\frac 1{\ell_1^2}  \left( {\rm Id}_1- e^{-\ell_1^2\triangle_1 }  \right)    a_{R,2}.
\end{split}\end{equation}
So we have
\begin{equation*}\begin{split}
  I^{R}_{112} &\leq  \frac 1{\ell_1^2} \int_{ ({\breve{C}}I^*)^c \times {\breve{C}}J^*}\left| \mathcal{ P}\circ\left( {\rm Id}_1- e^{-\ell_1^2\triangle_1 }  \right)    ( a_{R,2})
  (\mathbf{t},\mathbf{g})\right|d\mathbf{g}.
\end{split}\end{equation*}
Similarly as we have done for $I^{R}_{111} $, we get
\begin{equation*}\begin{split}
I^{R}_{112}
& \leq \frac C{\ell_1^2}  | J | ^{\frac 12}\int_{ ({\breve{C}}I^*)^c } d\mathbf{g}_1\int_{{\overline{C}} I} \left \|K^{(1)}(\mathbf{g}_1,\mathbf{g}'_1;
\mathbf{t})-K_{\ell_1^2,0}^{(1)}(\mathbf{g}_1,\mathbf{g}'_1; \mathbf{t})  \right\|_{L^2(\mathscr H_2)\rightarrow  L^2(\mathscr H_2)}\|(a_{R,2}) _{\mathbf{g}_1'}\|_{L^2(\mathscr H_2)  }d\mathbf{g}_1'\\
&\leq \frac C{\ell_1^2}      \gamma_1(R)^{-2} | J | ^{\frac 12}  \int_{{\overline{C}} I}  \|(a_{R,2}) _{\mathbf{g}_1'}\|_{L^2(\mathscr H_2)}d\mathbf{g}_1'\\
&\leq  C \frac 1{\ell_1^2}\gamma_1(R)^{-2}  | J | ^{\frac 12} \left(\int_{{\overline{C}} I}     \|(a_{R,2}) _{\mathbf{g}_1'}  \|^2_{L^2(\mathscr H_2)}d\mathbf{g}_1'\right)^{\frac 12} | I | ^{\frac 12}\\
&\leq \frac C {l(I)^2}\gamma_1(R)^{-2}  | R | ^{\frac 12}     \| a_{R,2}   \| _{L^2(\mathscr H_1 \times\mathscr H_2 )} .
\end{split}   \end{equation*}
Apply Journ\'e's covering Lemma
 \ref{lem:Journe}   to get
 \begin{equation*}\begin{split}
   \sum_{R\in m (\Omega)} I^{R}_{112}&\leq C \left(\sum_{R \in m_2(\Omega)}\gamma_1(R)^{-4}  | R | \right)^{\frac 12}   \left( \sum_{R= I \times J \in m (\Omega)}
   l(I)^{-4}\| ( \triangle_1^{N-1} \otimes \triangle_2^N )b_R   \|^2 _{L^2(\mathscr H )}\right)^{\frac 12}\\
   &\leq C |\Omega|^{\frac 12}|\Omega|^{-\frac 12}=C,
\end{split} \end{equation*}by the condition of a $(  2, N)$-atom for $\sigma_1=N-1,\sigma_2=N$ in \eqref{eq:atom}. Consequently,  $\sum_{R \in m (\Omega)} I^{R}_{11 }$ is bounded.

 Now let us estimate $I^{R}_{12}$. We decompose the identity operator on $L^2(\mathscr H_2)$ as follows:
\begin{equation*}\begin{split}
     {\rm Id}_2&=  \frac 2{\ell_2^2}\int_{0}^{\ell_2}s ({\rm Id}_2-e^{-s^2\triangle_2 })ds+\frac 2{\ell_2^2}\int_{0}^{\ell_2}s  e^{-s^2\triangle_2 }
     ds
    =: {\rm Id}_{\ell_2,1}+{\rm Id}_{\ell_2,2}.
\end{split}\end{equation*}
As in \eqref{eq:Id-l2}\begin{equation} \label{eq:id-2-2} \begin{split}
    {\rm Id}_{\ell_2,2}(a_R) & =  \left( \frac 2{\ell_2^2}\int_{0}^{\ell_2}s  e^{-s^2\triangle_2 } ds\right)\left({\rm Id}_1\otimes\triangle_2\right)
    a_{R,1} =  \left( \frac 2{\ell_2^2}\int_{0}^{\ell_2}s \triangle_2 e^{-s^2\triangle_2 } ds\right)    a_{R,1}\\
      & =\frac 1{ \ell_2^2}  \left( {\rm Id}_2- e^{-\ell_2^2\triangle_2 }  \right)    a_{R,1}.
\end{split}\end{equation}
If we write
\begin{equation*}\begin{split}
  \mathcal{ P}  &= \mathcal{ P}\circ {\rm Id}_{1}\circ {\rm Id}_{2}
    = \mathcal{P}\circ({\rm Id}_{\ell_1,1}+{\rm Id}_{\ell_1,2})\circ({\rm Id}_{\ell_2,1}+{\rm Id}_{\ell_2,2})
   \\&= \mathcal{P}\circ {\rm Id}_{\ell_1,1}\circ {\rm Id}_{\ell_2,1}+ \mathcal{P}\circ {\rm Id}_{\ell_1,1} \circ {\rm Id}_{\ell_2,2}+ \mathcal{P}\circ {\rm Id}_{\ell_1,2}\circ {\rm
   Id}_{\ell_2,1}+ \mathcal{P}\circ {\rm Id}_{\ell_1,2} \circ {\rm Id}_{\ell_2,2}\\
   &=:\mathcal{P}_1+\mathcal{P}_2+\mathcal{P}_3+\mathcal{P}_4,
\end{split}\end{equation*}then we have
\begin{equation*}\begin{split}
   I^{R}_{12} &\leq  \sum_{j=1}^4\int_{ ({\breve{C}}I^*)^c \times ({\breve{C}}J^*)^c}|\mathcal{ P}_j(a_R) (\mathbf{t},\mathbf{g})|d\mathbf{g}=:\sum_{j=1}^4 I^{R}_{12j}.
\end{split}\end{equation*}

Note that
\begin{equation*}\begin{split}
   I^{R}_{121}
   &= \frac 4{\ell_1^2 \ell_2^2}\int_{0}^{\ell_1} s_1 ds_1\int_{0}^{\ell_2} s_2ds_2 \int_{ ({\breve{C}}I^*)^c \times ({\breve{C}}J^*)^c}|\mathcal{ P}({\rm Id}_1-e^{-s_1^2\triangle_1}) ({\rm Id}_2-e^{-s_2^2\triangle_2 })(a_R) (\mathbf{t},\mathbf{g})|d\mathbf{g}\\
   &\leq \frac 4{\ell_1^2 \ell_2^2}\int_{0}^{\ell_1} s_1 ds_1\int_{0}^{\ell_2} s_2ds_2\int_{ ({\breve{C}}I^*)^c \times ({\breve{C}}J^*)^c}\left| (\mathcal{P}-\mathcal{P}e^{-s_1^2\triangle_1 }-\mathcal{P}e^{-s_2^2\triangle_2 }+\mathcal{P}e^{-s_1^2\triangle_1 } e^{-s_2^2\triangle_2 }) (a_R)\right|  d\mathbf{g} \\
   &= \frac 4{\ell_1^2\ell_2^2}\int_{0}^{\ell_1} s_1 ds_1  \int_{0}^{\ell_2}s_2ds_2\int_{ ({\breve{C}}I^*)^c \times ({\breve{C}}J^*)^c}\left |\int_{{\overline{C}}R}\left( K(\mathbf{g},\mathbf{g}'; \mathbf{t})-K_{s_1^2,0}(\mathbf{g},\mathbf{g}'; \mathbf{t})-K_{0,s_2^2}(\mathbf{g},\mathbf{g}'; \mathbf{t})\right.\right.\\
   &\hskip90mm\left.\left.+K_{s_1^2,s_2^2}(\mathbf{g},\mathbf{g}'; \mathbf{t}) \right) (a_R)(\mathbf{g}')d\mathbf{g}'\right| d\mathbf{g} \\
   &\leq\frac 4{\ell_1^2\ell_2^2}\int_{0}^{\ell_1} s_1 ds_1\int_{0}^{\ell_2}s_2ds_2\int_{{\overline{C}}R}d\mathbf{g}' \int_{ ({\breve{C}}I^*)^c \times ({\breve{C}}J^*)^c}\left | K(\mathbf{g},\mathbf{g}'; \mathbf{t})-
     K_{s_1^2,0}(\mathbf{g},\mathbf{g}'; \mathbf{t})-K_{0 ,s_2^2}(\mathbf{g},\mathbf{g}'; \mathbf{t})\right. \\
     &\hskip
    99mm\left. +K_{s_1^2,s_2^2}(\mathbf{g},\mathbf{g}'; \mathbf{t})  \right|  |  a_R (\mathbf{g}')|d\mathbf{g}  \\
   &\leq C  \gamma_1(R)^{-2}\gamma_2(R)^{-2}\frac 1{\ell_1^2\ell_2^2} \int_{0}^{\ell_1}s_1ds_1\int_{0}^{\ell_2}s_2ds_2\int_{{\overline{C}}R}  |
    a_R (\mathbf{g}')| d\mathbf{g}'\\
    & \leq C  \gamma_1(R)^{-2}\gamma_2(R)^{-2} | R | ^{\frac 12}     \| a_{R }   \| _{L^2(\mathscr H_1 \times\mathscr H_2 )}
\end{split}\end{equation*}
by using Proposition \ref{prop:K-estimate} again, where $\gamma_2(R):= \overline{C}\frac {l(J^*)}{l(J )}$.
For $ I^{R}_{122}$,
  by \eqref{id1} and \eqref{eq:id-2-2}, we find that
  \begin{equation*}\begin{split}
   I^{R}_{122} &=\int_{ ({\breve{C}}I^*)^c \times ({\breve{C}}J^*)^c}| \mathcal{P}\circ {\rm Id}_{\ell_1,1}\circ {\rm Id}_{\ell_2,2}(a_R) (\mathbf{t},\mathbf{g})|d\mathbf{g} \\
   &\leq \frac 2{\ell_1^2\ell_2^2} \int_{0}^{\ell_1}  s_1 ds_1  \int_{ ({\breve{C}}I^*)^c \times ({\breve{C}}J^*)^c}| \mathcal{P}({\rm Id}_1-e^{-s_1^2\triangle_1 })
   ({\rm Id}_2-e^{-\ell_2^2\triangle_2 })( a_{R,1}) (\mathbf{t},\mathbf{g})|d\mathbf{g} \\ &\leq\frac 2{\ell_1^2 \ell_2^2}\int_{0}^{\ell_1} s_1 ds_1\int_{{\overline{C}}R} d\mathbf{g}' \int_{ ({\breve{C}}I^*)^c \times ({\breve{C}}J^*)^c}\left |  K(\mathbf{g},\mathbf{g}'; \mathbf{t})-
     K_{s_1^2,0}(\mathbf{g},\mathbf{g}'; \mathbf{t})-K_{0, \ell_2 ^2}(\mathbf{g},\mathbf{g}'; \mathbf{t})\right. \\&\hskip 90mm\left. +K_{s_1^2,\ell_2^2}(\mathbf{g},\mathbf{g}'; \mathbf{t} )  \right|  | ( a_{R,1})(\mathbf{g}')|d\mathbf{g} \\
   &\leq C  \gamma_1(R)^{-2}\gamma_2(R)^{-2}\frac 1{\ell_1^2\ell_2^2 } \int_{0}^{\ell_1} s_1 ds_1\int_{{\overline{C}}R}  | ( a_{R,1})(\mathbf{g}')| d\mathbf{g}' \\&
 \leq C  \gamma_1(R)^{-2}\gamma_2(R)^{-2} l(J)^{-2}| R | ^{\frac 12}     \|  a_{R,1}  \| _{L^2(\mathscr H_1 \times\mathscr H_2 )}.
\end{split}\end{equation*}

Similarly, we have
 \begin{equation*}
    I^{R}_{123}
 \leq C  \gamma_1(R)^{-2}\gamma_2(R)^{-2} l(I)^{-2}| R | ^{\frac 12}     \|  a_{R,2}  \| _{L^2(\mathscr H_1 \times\mathscr H_2  )}.
 \end{equation*}
By $\triangle_1$ commuting $\triangle_2$,
 \begin{equation*}\begin{split}
   I^{R}_{124} &=\int_{ ({\breve{C}}I^*)^c \times ({\breve{C}}J^*)^c }  | \mathcal{P}\circ {\rm Id}_{\ell_1,2}\circ {\rm Id}_{\ell_2,2}(a_R) (\mathbf{t},\mathbf{g})|d\mathbf{g} \\
   &\leq \frac 1{\ell_1^2\ell_2^2}   \int_{ ({\breve{C}}I^*)^c \times ({\breve{C}}J^*)^c }
   \left| \mathcal{P}({\rm Id}_1-e^{-\ell_1^2\triangle_1 })  ({\rm Id}_2-e^{-\ell_2^2\triangle_2 })  (\triangle_1^{N-1} \otimes \triangle_2^{N -1})b_{R }\right |
   (\mathbf{t},\mathbf{g})d\mathbf{g}\\
&\leq\frac 1{\ell_1^2 \ell_2^2} \int_{{\overline{C}}R}d\mathbf{g}' \int_{ ({\breve{C}}I^*)^c \times ({\breve{C}}J^*)^c}\left | K(\mathbf{g},\mathbf{g}'; \mathbf{t})-K_{\ell_1^2,0}(\mathbf{g},\mathbf{g}'; \mathbf{t})-K_{0,\ell_2 ^2}(\mathbf{g},\mathbf{g}'; \mathbf{t})  + K_{\ell_1^2,\ell_2^2}(\mathbf{g},\mathbf{g}'; \mathbf{t})  \right|  \\ &\hskip 99mm  \left  |(\triangle_1^{N-1} \otimes \triangle_2^{N -1}) b_{R } (\mathbf{g}')\right|d\mathbf{g}.
  \end{split}\end{equation*}
  Thus
  \begin{equation*}\begin{split} I^{R}_{124} &\leq C  \gamma_1(R)^{-2}\gamma_2(R)^{-2}\frac 1{\ell_1^2\ell_2^2 }  \int_{{\overline{C}}R} \left |(\triangle_1^{N-1} \otimes \triangle_2^{N -1}) b_{R } (\mathbf{g}')\right| d\mathbf{g}'    \\
   & \leq C  \gamma_1(R)^{-2}\gamma_2(R)^{-2} l(I)^{-2}l(J)^{-2}| R | ^{\frac 12}     \| (\triangle_1^{N-1} \otimes \triangle_2^{N -1})b_{R }  \| _{L^2(\mathscr H_1 \times\mathscr H_2  )}.
\end{split}\end{equation*}

By applying  Journ\'e's covering Lemma
 \ref{lem:Journe},   H\"older's inequality and the condition of a $(
  2, N)$-atom in \eqref{eq:atom}, we get
 \begin{equation*}\begin{split}
    \sum_{R= I \times J\in m (\Omega)} I^{R}_{124} \leq &C \left(\sum_{R \in m_2(\Omega)}\gamma_1(R)^{-8}  | R | \right)^{\frac 14}  \left(\sum_{R \in m_1(\Omega)}\gamma_2(R)^{-8}  | R | \right)^{\frac 14}\\& \cdot \left( \sum_{R= I \times J \in m (\Omega)}
   l(I)^{-4}l(J)^{-4}\| (\triangle_1^{N-1} \otimes \triangle_2^{N -1})b_R   \|^2 _{L^2(\mathscr H_1 \times\mathscr H_2  )}\right)^{\frac 12}\\
\leq & C |\Omega|^{\frac 14}|\Omega|^{\frac 14}|\Omega|^{-\frac 12}=C.
 \end{split}\end{equation*}

 Similar bounds for  $\sum_{R \in m (\Omega)} I^{R}_{12j} $, $j=1,2,3$, hold.
Thus
\begin{equation*}\begin{split}
 \sum_{R\in m (\Omega)}  I^{R}_{12} &\leq   \sum_{R \in m (\Omega)} \sum_{j=1}^4 I^{R}_{12j} \leq C,
\end{split}\end{equation*}and so $ \sum_{R \in m (\Omega)}  I^{R}_{1 }$ is uniformly bounded.
The estimate for $\sum_{R \in m (\Omega)} I_2^R$ in \eqref{eq:I1-I2} follows by exchanging variables $ \mathbf{g}_1$ and $ \mathbf{g}_2$. The proposition is proved.
\end{proof}
\subsection{Atomic Hardy space }
We say that $f =\sum_{j=1}^\infty\lambda_j a_j$   is a  {\it $(
  2, N)$-atomic
representation} of $f$ if each $a_ j$ is a $(
  2, N)$-atom, $\sum_{j=1}^\infty |\lambda_j| <+\infty,$  and the sum
converges in $L^2(\mathscr H_1\times\mathscr H_2 )$. Set
\begin{equation*}
   \mathbb{H}^1_{  {at},N} (\mathscr H_1\times\mathscr H_2 )=\{f; f \text{ has a } (
  2, N)\text{-atomic
representation}\}
\end{equation*}
with the norm to be the
 infimum of $ \sum_{j=1}^\infty |\lambda_j|  $ taken over all possible representation of $f$.
Then  {\it atomic bi-parameter Hardy space} $H^1_{ {at},N} (\mathscr H_1\times\mathscr H_2 )$  is  defined as the completion of
$\mathbb{H}^1_{  {at},N} (\mathscr H_1\times\mathscr H_2 )$ under this norm.
By \cite[Theorem 2.9]{CDLWY} \cite[Proposition 3.5, 5.2, 5.3]{CFLY}, bi-parameter Hardy spaces   on
  stratified Lie groups   characterized by atomic
decompositions, or area functions, or maximal functions are all equivalent.
\begin{thm}\label{thm:atom-H} Suppose that $N>\max\{Q_1,Q_2\}/4$. Then,
$H^1_{ \text{at},N} (\mathscr H_1\times\mathscr H_2 ){=} H^1 (\mathscr H_1\times\mathscr H_2 )$  and they have equivalent  norms.
\end{thm}

It follows from   definition that  an element $f $ of $ H^1_{  {at},N} (\mathscr H_1\times\mathscr H_2 )$ can be written as
  \begin{equation}\label{eq:atomic-decomp}
     f =\sum_{j=1}^\infty\lambda_j a_j,
  \end{equation}which
converges as distributions,  where each $a_ j$ is a $(
  2, N)$-atom and $\sum_{j=1}^\infty |\lambda_j| <+\infty$.
  The norm   of $f\in H^1_{at, N}(\mathscr H_1\times\mathscr H_2 )$ is the
 infimum of $ \sum_{j=1}^\infty |\lambda_j|  $ taken over all possible decomposition of $f$.

\begin{prop}\label{prop:P} $\mathcal P$ can be extended to a bounded operator from $  H ^1(\mathscr H_1\times\mathscr H_2) $ to $H^1(\mathscr U)$.
  \end{prop}
  \begin{proof}
  For any finite sum $ S= \sum_{k=1}^{ M } \lambda_k  a_k $
 of $(2,N)$-atoms, by using Theorem \ref {prop0}, we have
 \begin{equation}\label{eq:atomic-bd}
   \| \mathcal P(S)\|_{H^1(\mathscr U)} \leq  \sum_{k=1}^{ M }  \|  \mathcal{P}(a_k )\| _{H^1(\mathscr U)}  \left|  \lambda_k  \right| \leq  C  \sum_{k=1}^{ M }     \left|
   \lambda_k  \right|.
 \end{equation}
  Since   finite sums
 of $(2,N)$-atoms are dense in $  H ^1(\mathscr H_1\times\mathscr H_2 ) $, we get the result.
 \end{proof}

\section{   Holomorphic atomic decomposition } \subsection{Maximal function of an $H ^1(\mathscr U) $ function }
The following estimate is a bi-parameter generalization of  \cite[lemma 8.5]{FS}.
 \begin{lem}\label{lem:FS} Suppose that $u$   satisfies   heat
equations $(\partial_{t_\alpha}+ \triangle_\alpha)u=0$, $\alpha=1,2$,  on $ \mathscr U$,  and $u^*\in L^p(\mathscr H_1\times\mathscr H_2)$.
  Then, there exists a constant $C>0$ only depending on $p,Q_1,Q_2$ such that
    \begin{equation*}
       |u(\mathbf{t}, \mathbf{g})|\leq C\|u^*\|_{L^p(\mathscr H_1\times\mathscr H_2)}\, t_1^{-\frac {Q_1}{2p }}t_2^{-\frac {Q_2}{2p }} ,
    \end{equation*}
    for any $(\mathbf{t}, \mathbf{g})\in \mathscr U$.
 \end{lem}
  \begin{proof}
 Since $|u(\mathbf{t}, \mathbf{g})|\leq u^*(\mathbf{h})$ whenever $ \|\mathbf{g}_1 ^{-1} \mathbf{h}_1\|_1^2  <t_1,  \|\mathbf{g}_2 ^{-1}  \mathbf{h}_2\|_2^2  <t_2$, we have
 \begin{align*}
    |u(\mathbf{t}, \mathbf{g})| ^p
    &\leq \frac 1{|B_1(\mathbf{g}_1,\sqrt t_1 )||B_2(\mathbf{g}_2,\sqrt t_2 )|}\int_{B_1(\mathbf{g}_1,\sqrt t_1 )\times B_2(\mathbf{g}_2,\sqrt t_2 )}|u^*(\mathbf{h})|^p  dV(\mathbf{h})\\
    &\leq C\|u^*\|_{L^p(\mathscr H_1\times\mathscr H_2)}^p\,t_1^{-\frac {Q_1}{2 }}t_2^{-\frac {Q_2}{2 }},
 \end{align*}
 for some constant $C$ only depending on $p,Q_1,Q_2$.
 \end{proof}
  \begin{prop}\label{prop:H1-H1} If $f\in H^1(\mathscr U)$, then  $\|f^*\|_{L^1(\mathscr H_1 \times\mathscr H_2)} \lesssim \|f\|_{H^1(\mathscr U )} $.
  \end{prop}
  \begin{proof} Note that $f$ is smooth on $\mathscr U$, since $ (\pi^{-1})^*f$ is holomorphic on $  \mathcal{ U}$ and $\pi$ is a diffeomorphism.   For fixed $\boldsymbol {\varepsilon }\in \mathbb{R}^2_+$, $f(\boldsymbol
  {\varepsilon } +\cdot,\cdot)\in H^1(\mathscr U )\cap C(\overline{\mathscr U})$ by definition. Apply     Proposition \ref{prop:heat-kernel-Hp} to $f(\boldsymbol
  {\varepsilon } +\cdot,\cdot) $ to get
 \begin{equation*}
 |f(\mathbf{{t}} +{\boldsymbol {\varepsilon }} ,\mathbf{g})|^q\leq\int_{\mathscr H_1 \times\mathscr H_2 }h_\mathbf{t}(\mathbf{g'}{}^{-1} \mathbf{g})
 |f(\boldsymbol {\varepsilon } ,\mathbf{g'})|^qd\mathbf{g'}
 \end{equation*} if we choose $0<q<1$.
 Then,  $r=\frac{1}{q}>1$,  $|f({\boldsymbol {\varepsilon }},\cdot)|^q\in L^r(\mathscr H_1 \times\mathscr H_2)$ and
 \begin{equation*} \begin{split}
 \sup\limits_{(\mathbf{t}  ,\mathbf{h})\in  \Gamma_{\mathbf{g }}} |f(\mathbf{{t}} +{\boldsymbol {\varepsilon }},\mathbf{h})|^q&\leq  \sup\limits_{(\mathbf{t}
 ,\mathbf{h})\in  \Gamma_{\mathbf{g }}} \int_{\mathscr H_1 \times\mathscr H_2  }h_\mathbf{t}(\mathbf{g'}{}^{-1} \mathbf{h}  )
 |f(\boldsymbol {\varepsilon },\mathbf{g'})|^qd\mathbf{g'} \lesssim
 {  M_2M_1\left({|f(\boldsymbol {\varepsilon },\cdot)|^q}\right)(\mathbf{g})},
 \end{split} \end{equation*}
where $   \Gamma_{\mathbf{g }}$
  is the  non-tangential region \eqref{eq:region} at $\mathbf{g}\in \mathscr H_1 \times\mathscr H_2$ and $M_\alpha$ is the Hardy-Littlewood maximal function on
  $\mathscr H_\alpha $. Therefore,
\begin{equation}   \label{maximal-ineq} \begin{split}
 \left\|\sup\limits_{(\mathbf{t}  ,\mathbf{h})\in  \Gamma_{\mathbf{g }}} |f(\mathbf{t} +\boldsymbol {\varepsilon },\mathbf{h})|^q\right\|_{L^r(\mathscr H_1
 \times\mathscr H_2
 )} &
 \lesssim \left\|M_2M_1\left({|f(\boldsymbol {\varepsilon }  ,\cdot)|^q}\right)(\mathbf{g})\right\|_{L^r(\mathscr H_1 \times\mathscr H_2 )}  \\
& \lesssim \|\left|f( \boldsymbol {\varepsilon },\cdot)\right|^q\|_{L^r(\mathscr H_1 \times\mathscr H_2 )}\leq  \|f\|_{H^1(\mathscr U )},
 \end{split}\end{equation}where implicit constants are independent of $f $ and $\boldsymbol {\varepsilon } $.
Letting $\varepsilon_1, \varepsilon_2 \rightarrow 0$ in \eqref{maximal-ineq}, we obtain
\begin{equation}   \label{maximal-ineq2} \begin{split} \left\|\sup\limits_{(\mathbf{t}  ,\mathbf{h})\in  \Gamma_{\mathbf{g }}} |f(\mathbf{t}
,\mathbf{h})| \right\|_{L^1(\mathscr H_1\times \mathscr H_2 )}&= \left\|\sup\limits_{(\mathbf{t}  ,\mathbf{h})\in  \Gamma_{\mathbf{g }}}
|f(\mathbf{t} ,\mathbf{h})|^q\right\|_{L^r(\mathscr H_1\times \mathscr H_2 )}\\ &
 \leq\lim_{\varepsilon_1,\varepsilon_2\rightarrow 0}  \left\|\sup\limits_{(\mathbf{t}  ,\mathbf{h})\in  \Gamma_{\mathbf{g }}} |f(\mathbf{t}
 +\boldsymbol {\varepsilon },\mathbf{h})|^q \right\|_{L^r(\mathscr H_1 \times\mathscr H_2  )}\\
& \lesssim \|f\|_{H^1(\mathscr U )}.
 \end{split}\end{equation}  by Fatou's theorem.
 \end{proof}

  \subsection{The  existence of   boundary
    distributions  }The following  existence of boundary
    distributions is a bi-parameter  generalization of   \cite[Theorem 8.8]{FS}.
 \begin{thm}\label{thm:boundary-distribution}
       Suppose that   $f\in H^1(\mathscr U)$. Then there exists $f^b\in
    \mathcal{S}'(\mathscr H_1\times\mathscr H_2 )$ such that $f(\varepsilon_1, \varepsilon_2,\cdot)\rightarrow f^b$ in $ \mathcal{S}'(\mathscr H_1\times\mathscr H_2  )$  as
    $\varepsilon_1, \varepsilon_2\rightarrow 0$.
 \end{thm}
   \begin{proof} Note that
   $f$ satisfies $(\partial_{t_\alpha}+ \triangle_\alpha)f=0$ on $ \mathscr U$ by Proposition \ref{prop:regular-sublap} and $f^*\in L^1(\mathscr H_1\times\mathscr H_2 )$ by Proposition \ref{prop:H1-H1}.
   For $\psi\in \mathcal{S}  (\mathscr H_1\times\mathscr H_2 )  $, let
   \begin{equation*}
      F(\mathbf{t}):=\int_{\mathscr H_1\times\mathscr H_2} f(\mathbf{t},
   \mathbf{g})\psi(\mathbf{g})dV(\mathbf{g})
   \end{equation*}
  for $\mathbf{t}\in \mathbb{R}^2_+$. The
   integral converges obviously by Lemma \ref{lem:FS}. Then
   \begin{equation*}
    \partial_{t_1}^{k_1} \partial_{t_2}^{k_2} F(\mathbf{t})=\int_{\mathscr H_1\times\mathscr H_2} \partial_{t_1}^{k_1} \partial_{t_2}^{k_2} f(\mathbf{t},
    \mathbf{g})\psi(\mathbf{g})d \mathbf{g} =\int_{\mathscr H_1\times\mathscr H_2}   f(\mathbf{t}, \mathbf{g})(-\triangle_1)^{k_1}(-\triangle_2)^{k_2}
    \psi(\mathbf{g})d \mathbf{g} ,
   \end{equation*}by integration by part.
   Thus,
   \begin{equation}\begin{split}\label{eq:derivative-F}
    |\partial_{t_1}^{k_1} \partial_{t_2}^{k_2} F(\mathbf{t})|
    &\leq\|f(\mathbf{t}, \cdot)\|_{L^\infty(\mathscr H_1\times\mathscr H_2) } \int_{\mathscr H_1\times\mathscr H_2
    }|(-\triangle_1)^{k_1}(-\triangle_2)^{k_2}
    \psi(\mathbf{g})|d \mathbf{g}\\
    & \lesssim  \|\psi\|_{k_1,k_2}\|f\|_{H^1(\mathscr U)}t_1^{-\frac {Q_1}{2  }}t_2^{-\frac {Q_2}{2 }}
  \end{split} \end{equation}by Lemma \ref{lem:FS}.
   In particular, $\partial_{t_1}^{k_1} \partial_{t_2}^{k_2} F(\mathbf{t})\rightarrow 0$ as $t_1\rightarrow +\infty $ or $t_2\rightarrow +\infty$. Hence,
   \begin{equation*}\begin{split}
 \partial_{t_1}^{k_1-1} \partial_{t_2}^{k_2} F(\mathbf{t})&=-\int_{t_1}^{+\infty}      \partial_{s_1}^{k_1} \partial_{t_2}^{k_2} F(s_1,t_2)ds_1 .
 \end{split}  \end{equation*}
   Taking $k_1= N_1:=\frac {Q_1}{2  }+1$, $N_1-1, \ldots,2$, we get
   \begin{equation*}
    |\partial_{t_1}^{2} \partial_{t_2}^{k_2} F(\mathbf{t})|  \lesssim \|\psi\|_{N_1,k_2}\|f\|_{H^1(\mathscr U)}t_1^{-1}t_2^{-\frac {Q_2}{2  }},
   \end{equation*}
   and so
    \begin{equation}\label{eq:derivative-F2}
    |\partial_{t_1}  \partial_{t_2}^{k_2} F(\mathbf{t})| \leq |\partial_{t_1}  \partial_{t_2}^{k_2} F(1, {t}_2)|+\int_{t_1}^1|\partial_{t_1}^2
    \partial_{t_2}^{k_2} F(s_1,t_2)|ds_1  \lesssim \|\psi\|_{N_1,k_2}\|f\|_{H^1(\mathscr U)}(1+|\log t_1|)t_2^{-\frac {Q_2}{2  }}
   \end{equation}
  by using \eqref{eq:derivative-F} for $k_1=1,t_1=1$. Apply the same argument to $t_2$ to get
 \begin{equation*}
    |\partial_{t_1}  \partial_{t_2}  F(\mathbf{t})|  \lesssim \|\psi\|_{N_1,N_2}\|f\|_{H^1(\mathscr U)}(1+|\log t_1|)(1+|\log t_2|),
   \end{equation*}with $N_2:=\frac {Q_2}{2  }+1$.
   We also have
   \begin{equation*}
      |\partial_{t_1}    F( {t}_1,1)|  \lesssim \|\psi\|_{N_1,N_2}\|f\|_{H^1(\mathscr U)}(1+|\log t_1|) ,
   \end{equation*} by using \eqref{eq:derivative-F2} for $k_2=0,t_2=1$.
   Therefore,
   \begin{equation*}
 \lim\limits_{\varepsilon _1 \rightarrow 0}   F(\varepsilon _1, 1 ) = F(1,1)  -\lim\limits_{\varepsilon _1\rightarrow 0 }\int_{\varepsilon_1}^1  \partial_{t_1}
 F(t_1,1) dt_1
   \end{equation*}exists and is bounded by $ \|\psi\|_{N_1,N_2}\|f\|_{H^1(\mathscr U)}$. So does $\lim\limits_{ \varepsilon _2\rightarrow 0}   F( 1,\varepsilon_2) $. At last, we see
   that
   \begin{equation*}
 \lim\limits_{\varepsilon_1, \varepsilon _2\rightarrow 0}  F(\varepsilon_1, \varepsilon _2)= - F(1,1)+\lim\limits_{\varepsilon _1 \rightarrow 0}   F(\varepsilon _1, 1 )   +\lim\limits_{ \varepsilon _2\rightarrow 0}   F( 1,\varepsilon_2)
   +\lim\limits_{\varepsilon_1, \varepsilon _2\rightarrow 0}\int_{\varepsilon_1}^1\int_{\varepsilon_2}^1 \partial_{t_1}  \partial_{t_2}   F(t_1,t_2) dt_1 dt_2
   \end{equation*}exists and is bounded by $ \|\psi\|_{N_1,N_2}\|f\|_{H^1(\mathscr U)}$. So the limit defines a distribution $f^b$ on $\mathscr H_1\times\mathscr H_2  $,  and $\lim\limits_{\varepsilon_1,
   \varepsilon_2\rightarrow
   0}f(\boldsymbol {\varepsilon },\cdot)=f^b$ as distributions.
    \end{proof}
\begin{cor}\label{prop:heat-convolution} If $f\in H^1(\mathscr U)$, then  $f^b\in  H ^1(\mathscr H_1\times\mathscr H_2) $ and $f(\mathbf{t} , \cdot)=h_\mathbf{t}
 * f^b$.
  \end{cor}  \begin{proof}
 Consider $f_{ k }(\mathbf{t} , \mathbf{g} ):=f ( \mathbf{t }+\boldsymbol {\varepsilon }_k, \mathbf{g }) $,  where    $\boldsymbol {\varepsilon
 }_k:=(\varepsilon_k,\varepsilon_k)$, $  {\varepsilon_k } :=1/k$.
By definition, we have $f_{ k }\in H^1(\mathscr U )$ and is smooth on $\overline{\mathscr U}$.
Let $F_k  (  \mathbf{h }
    ):=f_k  ( \mathbf{0} , \mathbf{h}
    )\in L^1(\mathscr H_1\times \mathscr H_2)$. Then,   we have
 \begin{equation*}
     {f}_{ k }(\mathbf{t} , \mathbf{g})= h_{\mathbf{t } }* F_k  (   \mathbf{g}
    )
 \end{equation*}
 by Proposition \ref{prop:heat-kernel-H1}. Denote $\breve{h}_{  {\mathbf{t} } ;\mathbf{g}}( {\mathbf{h}}):= h_{ {\mathbf{t} }} ( {\mathbf{h}}^{-1}\mathbf{g} )$, which also   belongs to $
\mathcal{S}(\mathscr H_1\times \mathscr H_2) $.  Then
   \begin{equation*} f (\mathbf{t} , \mathbf{g})=
    \lim_{k\rightarrow\infty} {f}_{ k }(\mathbf{t} , \mathbf{g})= \lim_{k\rightarrow\infty}  h_{\mathbf{t } }* F_k  (  \mathbf{g }
    ) =\lim_{k\rightarrow\infty}  \langle   F_k   , \breve{h}_{ \boldsymbol {\varepsilon } ;\mathbf{g}}  \rangle = \langle    f^b  , \breve{h}_{ \boldsymbol {\varepsilon } ;\mathbf{g}}  \rangle = (h_{\mathbf{t } }* f^b )  (  \mathbf{g }    )
 \end{equation*}
 by  continuity of $f $ at $(\mathbf{t} , \mathbf{g})\in \mathscr U$ and the convergence of distributions $ F_k \rightarrow f^b$ by Theorem \ref{thm:boundary-distribution}.
   \end{proof}
\begin{prop} \label{lem-atom} \cite[Theorem 2.7]{FS} Let $u\in \mathcal{S} '(\mathcal{N})$ on a  homogeneous group $\mathcal{N}$.
 If there exists $\phi\in \mathcal{S}(\mathcal{N}) $ with   $\int _{\mathcal{N}}\phi=1$
  such that $\sup_{t}|\phi_t*u|\in L^1(\mathcal{N})$, then $u\in L^1(\mathcal{N})$.
\end{prop}  $ \mathscr H_1\times \mathscr H_2 $ is a homogeneous group with dilation $\hat \delta_r(\mathbf{g}_1,\mathbf{g}_2):=
(\delta_r^{(1)}(\mathbf{g}_1),\delta_r^{(2)}(\mathbf{g}_2))$. We can apply this proposition to  $ \mathscr H_1\times \mathscr H_2 $ and   the heat kernel to obtain

\begin{cor} \label{cor:L1} For a distribution $ u\in H^1 (\mathscr H_1\times \mathscr H_2)$, we have  $ u\in L^1 (\mathscr H_1\times \mathscr H_2)$.
\end{cor}

 \subsection{Proof of Theorem \ref{thm:atom} }
By Corollary \ref{prop:H1-H1} and Corollary \ref{cor:L1}, we see that $f^b\in  H ^1(\mathscr H_1\times\mathscr H_2) \cap  L^1(\mathscr H_1\times\mathscr
H_2) $.
 By applying Theorem \ref{thm:atom-H} to   $f^b
  $, we obtain
  an atomic decomposition
$f^b=
    \sum_{k }  \lambda_k  a_k
$ with $\|f^b\| _{H ^1(\mathscr H_1 \times\mathscr H_2 ) }\approx  \sum_{k } |\lambda_k | $.
Since the summation   converges in $ \mathcal{S}'(\mathscr H_1\times \mathscr H_2)$, we get
\begin{equation}\label{eq:f-epsilon-atom}\begin{split}
   f({\boldsymbol {\varepsilon }},\mathbf{ g})&=h_{{\boldsymbol {\varepsilon }}} *f^b( \mathbf{ g})=\langle f^b ,\breve{h}_{ \boldsymbol {\varepsilon }
   ;\mathbf{g}}\rangle=  \sum_{k }   \left \langle \lambda_k a_k ,\breve{h}_{ \boldsymbol {\varepsilon } ;\mathbf{g}}\right\rangle   =   \sum_{k } \lambda_k  {h}_{
   \boldsymbol {\varepsilon }  } * a_k( \mathbf{ g}).
  \end{split}\end{equation}

  Note that   for $ a_k  \in H^1 (\mathscr H_1\times \mathscr H_2)$, we have
$h_{\boldsymbol {\varepsilon }} * a_k \in H^1(\mathscr H_1\times \mathscr H_2 )$ with
 \begin{equation}\label{eq:atoms-bound}
  \|  h_{\boldsymbol {\varepsilon }} * a_k  \|_{ H^1(\mathscr H_1\times \mathscr H_2 )}\leq  \|    a_k  \|_{ H^1(\mathscr H_1\times \mathscr H_2 )} \leq C_3,
 \end{equation}for some absolute constant $C_3>0$. This is simply because
 \begin{equation*}
    \left [h_{\mathbf{t} }*( h_{\boldsymbol {\varepsilon }} * a_k  )\right] ^*(\mathbf{g}) =\sup\limits_{(\mathbf{t}  ,\mathbf{h})\in  \Gamma_{\mathbf{g }}} \big|h_{\mathbf{t}+ \boldsymbol {\varepsilon } }* a_k  \big|(\mathbf{h}
    )\leq  (h_{\mathbf{t} }*a_k )^*(\mathbf{g}),
 \end{equation*}and $H^1$-norms of $(  2,N)$-atoms are uniformly bounded  by Theorem \ref{thm:atom-H}.
  Thus for fixed ${\boldsymbol {\varepsilon }} >0$, $ h_ {\boldsymbol {\varepsilon }} *  a_k \in L^1(\mathscr H_1\times \mathscr H_2)$   and   $\|h_{\boldsymbol
{\varepsilon }} *  a_k \|_{L^\infty(\mathscr H_1\times \mathscr H_2)}$ is bounded by Lemma \ref{lem:FS}. Thus, $h_{\boldsymbol {\varepsilon }} *  a_k \in
L^2(\mathscr H_1\times \mathscr H_2)$ with $L^2$ norm bounded by a constant independent of $k $. Moreover,
 $\sum_k \lambda_k h_{\boldsymbol {\varepsilon }} *  a_k  $ is  convergent in $ L^2(\mathscr H_1\times \mathscr H_2)$ by the convergence of $ \sum_{k } |\lambda_k
 |^2$, which follows from the convergence of $ \sum_{k } |\lambda_k
 | $.
So we can apply the Cauchy-Szeg\H{o} projection  $ \mathcal P$ in both sides of \eqref{eq:f-epsilon-atom} to get
 \begin{equation} \label{S=P2}
     \mathcal P(f({\boldsymbol {\varepsilon }},\cdot) ) (\mathbf{t},\mathbf{g})  =
       \sum_{k } \lambda_k \mathcal{P}\left(h_{\boldsymbol {\varepsilon }} * a_k \right ) (\mathbf{t},\mathbf{g}) .
  \end{equation}

   For fixed $ {\boldsymbol {\varepsilon }}>0$, $f({\boldsymbol {\varepsilon }} , \cdot) \in L^1(\mathscr H_1\times \mathscr H_2)$ and is bounded by Proposition
 \ref{lem:bound} or Lemma \ref{lem:FS}. Thus $f({\boldsymbol {\varepsilon }} , \cdot) \in
 L^2(\mathscr H )$.  Note that $f({\boldsymbol {\varepsilon }} , \cdot)$ is the boundary value of   $f( {\boldsymbol {\varepsilon }} +\cdot, \cdot  )$, which is smooth on
     $\overline{\mathscr  U}$. Thus
  \begin{equation*}
    \mathcal  P(f({\boldsymbol {\varepsilon }} , \cdot) )(\mathbf{t},\mathbf{g})=f(\mathbf{t}+{\boldsymbol {\varepsilon }}, \mathbf{g})
 \end{equation*}
  by the reproducing formula
of the Cauchy-Szeg\H{o} projection, and so
 \begin{equation}\label{eq:f-limit}
   \lim_{{\boldsymbol {\varepsilon }} \rightarrow 0}  \mathcal  P(f({\boldsymbol {\varepsilon }} , \cdot) )(\mathbf{t},\mathbf{g})= f(\mathbf{t}, \mathbf{g})
 \end{equation}for fixed $\mathbf{t},\mathbf{g}$, by the smoothness of $f$.

Let us show the series in the right hand side of \eqref{S=P2} converges uniformly for  $ {\boldsymbol {\varepsilon }} \in (0,1)\times (0,1) $.
We claim  that for fixed $\mathbf{t},\mathbf{g}$ and any given $\eta>0$, there exists positive integer $M $ such that
 \begin{equation} \label{S=P0}
   \left | \mathcal P(f({\boldsymbol {\varepsilon }}, \cdot) ) (\mathbf{t},\mathbf{g}) -\sum_{ k=1}^M  \lambda_k  \mathcal    P\left(h_{\boldsymbol {\varepsilon }}*
   a_k \right ) (\mathbf{t},\mathbf{g})
   \right|< \eta
 \end{equation} holds uniformly for  $ {\boldsymbol {\varepsilon }}\in (0,1)^2$.
    This is because
 \begin{equation*}\begin{split}
   \sum_{ k>M}|\lambda_k | \left |\mathcal   P\left(h_{\boldsymbol {\varepsilon }} * a_k  \right ) (\mathbf{t},\mathbf{g})\right| &
    \leq   C  \sum_{ k>M }  \left|  \lambda_k \right|\|\mathcal P(h_{\boldsymbol {\varepsilon }} * a_k )\|_{H^1(\mathscr U)} t_1^{-\frac {Q_1}{2 }}t_2^{-\frac
    {Q_2}{2 }}
\\
   &\leq C \|\mathcal P\|_{H ^1(\mathscr H_1\times \mathscr H_2 ) \to H^1(\mathscr U)} t_1^{-\frac {Q_1}{2 }}t_2^{-\frac {Q_2}{2 }}
    \sum_{ k>M } \left|  \lambda_k \right| \| h_{\boldsymbol {\varepsilon }} * a_k \| _{H ^1(\mathscr H_1\times \mathscr H_2 ) }
      \\&
       \leq C C_3\|\mathcal P\|_{H ^1(\mathscr H_1\times \mathscr H_2 ) \to H^1(\mathscr U)}  t_1^{-\frac {Q_1}{2 }}t_2^{-\frac {Q_2}{2 }}\sum_{ k>M }
      \left|  \lambda_k \right|
     \leq \eta
 \end{split} \end{equation*} if $M $ is large, by Proposition \ref{prop:P} and \eqref{eq:atoms-bound}.
  Note that for any $F \in L^2(\mathscr H_1\times
 \mathscr H_2) $, $\mathcal P(F) \in H^2(\mathscr U)$ and so
 \begin{equation*}
    |\mathcal  P\left( F\right ) (\mathbf{t},\mathbf{g})|\leq  C t_1^{-\frac {Q_1}{4 }}t_2^{-\frac {Q_2}{4 }} \|\mathcal P(F)\|_{H^2(\mathscr U)}\leq  C
    t_1^{-\frac {Q_1}{4 }}t_2^{-\frac {Q_2}{4 }} \| F
    \|_{L^2(\mathscr H_1\times \mathscr H_2)},
 \end{equation*}by using Proposition \ref{lem:bound}. Consequently,  for any fixed $(\mathbf{t},\mathbf{g})\in \mathscr U$, $F\mapsto \mathcal  P\left( F\right ) (\mathbf{t},\mathbf{g})$ is a continuous linear functional on
 $L^2(\mathscr H_1\times \mathscr H_2)$.
 Letting $\varepsilon_1, \varepsilon _2 \rightarrow 0$ in \eqref{S=P0}, we get
\begin{equation} \label{S=P}
   \left |  f (\mathbf{t},\mathbf{g})  -\sum_{ k=1}^M \lambda_k    \mathcal   P\left(  a_k  \right ) (\mathbf{t},\mathbf{g}) \right| \leq \eta
 \end{equation}
by \eqref{eq:f-limit} and $h_ {\boldsymbol {\varepsilon }} * a_k  \rightarrow  a_k  $ in $L^2 (\mathscr H_1\times \mathscr H_2)$. At last,  we
get   holomorphic atomic decomposition by letting $M\rightarrow+\infty$. \qed

\section{  Parabolic  maximum principle and  parabolic version of subharmonicity}
 \label{subharmonicity}
\subsection{    Parabolic  maximum principle}
Maximum principle for the  heat equation   was already used by Folland-Stein   \cite[Proposition 8.1]{FS} in the theory of Hardy spaces     on
 homogeneous groups. We give its proof here since we need the proof for functions nonsmooth somewhere.

\begin{prop}  \label{prop:max}  {\rm (Maximum principle)} Let $D$ be a bounded domain in $\mathscr H_\alpha$ and $\Omega=(0,T)\times D $ for $T>0$. Suppose
that
 $v\in C^2( \overline{{\Omega}}) $,  $v|_{[0,T)\times \partial D }\leq 0$, $v|_{ \{0\}\times D }\leq 0$ and $\mathcal{L}_\alpha v\geq 0$ in $\Omega$. Then $v\leq
 0$ in $\Omega$.
\end{prop}
\begin{proof} It is proved as the classical case \cite[Lemma 2.1]{Lieb}.  If   replace $v$ by $v-\kappa_1t_\alpha-\kappa_2$ for some $ \kappa_1,\kappa_2>0$, we may
assume $v|_{[0,T)\times\partial\Omega }< 0$, $v|_{ \{0\}\times\Omega
}<0$ and $\mathcal{L}v> 0$.
Suppose that $v>0$ somewhere in $ {\Omega}$. Let
\begin{equation*}
   t_\alpha^*:=\inf\{t_\alpha\mid v(t_\alpha,\mathbf{g}_\alpha)>0 \text{  for some  } \mathbf{g}_\alpha\in \overline{\Omega}\}.
\end{equation*}
  By continuity
and negativity of $v$ on the boundary $[0,T)\times\partial\Omega\cup  \{0\}\times\Omega $, we see that $v(t^*_\alpha, \mathbf{g}^*_\alpha)=0$ for some $(t_\alpha^*,\mathbf{ g}_\alpha^*) \in \Omega$. We must have $v(t_\alpha,\mathbf{g}_\alpha )<0$ for $0<t_\alpha<t_\alpha^*$ and
$\mathbf{g}_\alpha\in\Omega$, and so
\begin{equation}\label{eq:maximum-1}
   \partial_{t_\alpha}v(t_\alpha^*, \mathbf{g}_\alpha^*)\geq 0.
\end{equation}

On the other hand, $v(t_\alpha^*,\cdot ) $ attains its maximum at $\mathbf{g}_\alpha^*$, which implies that $X_{\alpha j}v(t_\alpha^*, \mathbf{g}_\alpha^*)=0$ and
\begin{equation}\label{eq:maximum-2}
   X_{\alpha j}^2v(t_\alpha^*,\mathbf{ g}_\alpha^*)=\left.\frac {d^2}{ds^2} v\left(t_\alpha^*,\mathbf{g}_\alpha^* \gamma_s\right)\right|_{s=0}\leq 0,
\end{equation}$j=1,\ldots, 2n_\alpha$, where $\gamma_s=( \dots,0,s,0,\ldots )$ (only the $(j+1)$-th entry is nontrivial) is the Lie subgroup of one parameter associated to the vector field $X_{\alpha j}$. Consequently, we
get $\mathcal{L}_\alpha v(t_\alpha^*, \mathbf{g}_\alpha^*)\leq 0$ by \eqref{eq:maximum-1}-\eqref{eq:maximum-2}, which contradicts  to $\mathcal{L}_\alpha v> 0$
in $\Omega$. Thus $v-\kappa_1t_\alpha-\kappa_2\leq 0$. Now letting $\kappa_1,\kappa_2\rightarrow 0+$, we get the result.
 \end{proof}

\begin{proof}[Proof of Proposition \ref{prop:heat-kernel-H1}] For  $f\in H^p( \mathscr U ) $, consider $f_{ k }(\mathbf{t} , \mathbf{g} ):=f ( \mathbf{t }+\boldsymbol {\varepsilon }_k, \mathbf{g }) $   as in the proof of Corollary \ref{prop:heat-convolution}, where    $\boldsymbol {\varepsilon
 }_k =(\varepsilon_k,\varepsilon_k)$, $  {\varepsilon_k } :=1/k$. It is
smooth on $\overline{\mathscr U}$ and satisfies   heat equations $ \mathcal{L}_\alpha  {{f}}_{ k }=0$, $\alpha=1,2$, by Proposition \ref{prop:regular-sublap}.
$f ( \boldsymbol {\varepsilon }_k,\cdot)\in L^p(\mathscr H_1\times \mathscr H_2)$ by definition. On the other hand,
$f ( \boldsymbol {\varepsilon }_k,\cdot)\in L^\infty(\mathscr H_1\times \mathscr H_2)$ by Proposition \ref{lem:bound}. Thus,
  $f ( \boldsymbol {\varepsilon }_k,\cdot)\in L^1(\mathscr H_1\times \mathscr H_2)$ since $p\geq 1$. Let
 \begin{equation*}
    \widetilde {f}_{ k }(\mathbf{t} , \mathbf{g})=[ h_{\mathbf{t } }* f_k  ( \mathbf{0} , \cdot
    )](\mathbf{g}),
 \end{equation*}
 which is also smooth in $\overline{\mathscr U}$ with  $\widetilde{f}_k  ( \mathbf{0} , \cdot
    )=f_k  ( \mathbf{0} , \cdot
    )$ by $h\in  \mathcal{S} (\mathscr H_1\times\mathscr H_2 )$,   and satisfies
   $ \mathcal{L}_\alpha \widetilde{{f}}_{ k }=0$ on $\mathscr U$.

 To show $  {f}_{ k }=\widetilde {f}_{ k }$  on $ \mathscr U $, we need to apply    maximum principle twice successively to   $\mathbf{g}_1$ and $\mathbf{g}_2$.  At first, we show that\begin{equation}\label{eq:id-2=0}
 {f}_{ k }(  t_1, 0, \mathbf{g} )= \widetilde{f}_{ k }( t_1, 0, \mathbf{g} )
   \end{equation}for any $t_1 >0$, $\mathbf{g} \in \mathscr H_1\times\mathscr H_2$.
Denote
   \begin{equation*}\begin{split}
       f_  { k(0,  \mathbf{g}_2)}(t_1,\mathbf{g}_1):&=f_k ( t_1 ,  0, \mathbf{g}_1,\mathbf{g}_2) ,\\   \widetilde f_ { k(0,  \mathbf{g}_2)}  (t_1,\mathbf{g}_1):&=\widetilde{f}_k ( t_1 , 0,
      \mathbf{g}_1,\mathbf{g}_2),
   \end{split}  \end{equation*}as functions on  $ {\mathscr U}_1$, for fixed $\mathbf{g}_2$.
      To apply  maximum principle to real components, write $ f_ { k(0,  \mathbf{g}_2)}= f_ { k(0,  \mathbf{g}_2)}^{ 1 } +\mathbf{i } f_ { k(0,  \mathbf{g}_2)}^{ 2 }   $ and  $ \widetilde f_{k (0,  \mathbf{g}_2)}= \widetilde f_  { k(0,  \mathbf{g}_2)}^{ 1 } +\mathbf{i
      } \widetilde f_{k (0,  \mathbf{g}_2)}^{ 2 }   $. Then,
 \begin{equation*}\left.\left[ f_{ k (0,  \mathbf{g}_2)}^{ \beta }-  \widetilde f_{ k (0,  \mathbf{g}_2)}^{ \beta  }\right]\right|_{\{0\}
   \times\mathscr H_1 }=0 \quad\operatorname{and} \quad
   \mathcal{L}_1\left[ f_ {k (0,  \mathbf{g}_2)}^{ \beta }-  \widetilde f_ { k(0,  \mathbf{g}_2)}^{ \beta }\right]=0,
 \end{equation*} on $\mathscr U_1$, $\beta=1,2$.
 We claim that for given $T>0$ and $\eta >0$, there exists $r_0>0$ such that
 \begin{equation}\label{eq:claim}
   \left | f_{ k (0,  \mathbf{g}_2)}^{ \beta }-  \widetilde f_{k (0,  \mathbf{g}_2)}^{ \beta }\right|\leq \eta\qquad \text{ on  }\quad [0,T)\times \partial B_1(\mathbf{0}_1,r) ,
 \end{equation}for $r\geq r_0$.
   Then we can apply maximum principle  in Proposition \ref{prop:max} for $\mathcal{L}_1 $ to
$
      f_{k (0,  \mathbf{g}_2)}^{\beta }-  \widetilde f_{ k (0,  \mathbf{g}_2)} ^{ \beta }- \eta
$
    to get $  f_ {k (0,  \mathbf{g}_2)}^{ \beta } - \widetilde f_{ k (0,  \mathbf{g}_2)}^{ \beta }  \leq \eta $ on $[0,T)\times   B_1(\mathbf{0}_1,r)$. Consequently, by  letting $r\rightarrow\infty$, $T\rightarrow\infty$ and $\eta\rightarrow 0$,
    we get
     \begin{equation*}
       f_{ k (0,  \mathbf{g}_2)}^{ \beta } \leq  \widetilde f_ { k(0,  \mathbf{g}_2)}^{ \beta }
   \end{equation*} on $\mathscr U_1$.  The same argument gives us the reverse inequality. Thus $ f_ {k (0,  \mathbf{g}_2)}^{ \beta }=  \widetilde f_{ k (0,  \mathbf{g}_2)}^{ \beta }
  $ on $\mathscr U_1$, i.e. \eqref{eq:id-2=0} holds.

   Now fix $t_1>0, \mathbf{g}_1\in \mathscr H_1$, applying maximum principle   for $\mathcal{L}_2 $ to functions on  $ {\mathscr U}_2$
   \begin{equation*}\begin{split}
        f_{ k (t_1 ,    \mathbf{g}_1)}(t_2,\mathbf{g}_2):&= {f}_k ( t_1 , t_2 ,
   \mathbf{g}_1,\mathbf{g}_2) ,\\  \widetilde f_ {k (t_1 ,    \mathbf{g}_1)}(t_2,\mathbf{g}_2):&= \widetilde{{f}}_k ( t_1 , t_2 , \mathbf{g}_1,\mathbf{g}_2),
   \end{split}  \end{equation*}
     as above, we find that  $f_{ k (t_1 ,    \mathbf{g}_1)}= \widetilde{f}_{ k (t_1 ,    \mathbf{g}_1)}$ on  $ {\mathscr U}_2$. Thus,
    $f_{k}= \widetilde{f}_{k}$ on $ \mathscr U $, i.e.
\begin{equation}\label{eq:k-reproducing}
    {f} \left( \mathbf{t} +\boldsymbol {\varepsilon }_k ,\mathbf{g}\right)=\int_{\mathscr H }h_{\mathbf{t }} (\mathbf{h}{}^{-1}  \mathbf{g})f \left(
    \boldsymbol {\varepsilon }_k,
    \mathbf{ h}\right)d\mathbf{h}.
 \end{equation}

 Since   $ L^p(\mathscr H_1\times \mathscr H_2)$  for $p> 1$ is reflexive, there exists a
subsequence  of $\{f  (
    \boldsymbol {\varepsilon }_k,
    \cdot )\}$, which is weakly convergent to some $\widetilde{f} \in L^p(\mathscr H_1\times \mathscr H_2)$
by Banach-Alaoglu theorem. We must have $\widetilde{f}(\mathbf{h})= f (\mathbf{0}, \mathbf{h})  $ by the continuity of $f$ on $\overline{\mathscr U} $.
If $p= 1$,
 we  apply  Banach-Alaoglu theorem to the dual space of $C( \mathscr H_1\times \mathscr H_2 )$, which contains  $L^1(\mathscr H_1\times \mathscr
 H_2)$,
to obtain  a subsequence of $\{f  (
    \boldsymbol {\varepsilon }_k,
    \cdot )\}$   weakly converging to a bounded  measure on $ \mathscr H_1\times \mathscr H_2  $, which is $ f (\mathbf{0}, \mathbf{h})d \mathbf{h} $ by the continuity of $f$ on $\overline{\mathscr U} $. Taking limit in
\eqref{eq:k-reproducing} as $k\rightarrow +\infty$, we get the result.

To prove the  boundary condition in the claim (\ref{eq:claim}), note that as in the proof of Proposition \ref {lem:FS},   for $\|\mathbf{g}_1\|_1=r\geq r_0/2$, $t_1\in [0,T]$ and $t_2=0$, we have
 \begin{equation*}\begin{split}
    | {f}_{k}(\mathbf{t}, \mathbf{g})| &= | {f} (\mathbf{t}+\boldsymbol {\varepsilon }_k, \mathbf{g})| \leq \left(\frac 1{|B_1(\mathbf{g}_1,\sqrt {  {\varepsilon }_k })||B_2(\mathbf{g}_2,\sqrt {\varepsilon }_k )|}\int_{B_1(\mathbf{g}_1,\sqrt {  {\varepsilon }_k })\times B_2(\mathbf{g}_2,\sqrt {\varepsilon }_k )}|f^*(\mathbf{h})|^p d \mathbf{h}\right)^{\frac 1p} \\ &\leq
    C_{ k}\left(\int_{B_1(\mathbf{0}_1,r_0/4 )^c\times \mathscr H_2} |f^* (\mathbf{h}) |^p d \mathbf{h}\right)^{\frac 1p} \leq \frac\eta 4 ,
\end{split}\end{equation*}
for sufficiently large $r_0>0$, by $f^*\in L^p(\mathscr H_1\times \mathscr
 H_2)$, where $C_{ k}$ is a constant only depending on $\varepsilon_k$, $Q_1$, $Q_2$. Consequently, we have
\begin{equation}\label{eq:estimate1}
  |\widetilde {f}_{ k }(\mathbf{t} , \mathbf{g})|= \left |\int_{\mathscr H_1\times \mathscr H_2 }h_\mathbf{t}(\mathbf{h}{}^{-1}  \mathbf{g})f \left ( \boldsymbol {\varepsilon }_k,
  \mathbf{h}\right)d\mathbf{h}\right|\leq \frac\eta2
\end{equation}for $\|\mathbf{g}_1\|_1\geq r\geq  r_0$ and $t_1,>  {\varepsilon_k }, t_2=0$ with sufficiently large $r_0$, since the heat kernel decays rapidly.
  \end{proof}

\subsection{    Parabolic version of subharmonicity}
We need the following parabolic version of subharmonicity of $|u|^p$ (cf. \cite[Section 3.2.1 in Chapter 7]{St70} for the Euclidean case).
 \begin{prop}  \label{prop:sub}  Suppose $f$ is holomorphic on $\mathscr U_\alpha  $. Then for any  $ p >0$, we have
 \begin{equation}\label{eq:p-subparabolic0}
   \mathcal{ L}_\alpha|f|^p( t_\alpha,\mathbf{g}_\alpha  )\geq 0,
 \end{equation}
 for $( t_\alpha,\mathbf{g}_\alpha  )\in\mathscr U_\alpha $ with $f( t_\alpha,\mathbf{g}_\alpha  )\neq 0$.
\end{prop}
\begin{proof}Since $\tau_{\mathbf{h}_\alpha}^*  f $ is also a holomorphic  function for fixed $\mathbf{h}_\alpha  \in\mathscr H_\alpha$ and $\mathcal{ L}_\alpha$ is also invariant
under
translations, we only need to show (\ref{eq:p-subparabolic0}) at point $(t_\alpha,\mathbf{0}_\alpha )$.

Note that
    \begin{equation*}
     X_{\alpha j}|f|^p =\frac p2\left (f\cdot \overline{ f}\right)^{\frac p2-1}( X_{\alpha j} f \cdot \overline{f} + f\cdot\overline{ X_{\alpha j} f} ),
 \end{equation*} and $ X_{\alpha j} f \cdot \overline{f} + f \cdot \overline{X_{\alpha j} f}
 =2{\rm Re} \left(X_{\alpha j} f \cdot  \overline{f}\right)$.
  Then, we have
   \begin{equation}\label{eq:p-subharmonic00}\begin{split}
   \sum_{j=1}^{2n_\alpha}  X_{\alpha j} ^2|f|^p=&\frac p2\left(\frac p2-1\right) |f|^{  p -4}4 \sum_{j=1}^{2n_\alpha} \left({\rm Re} (X_{\alpha j} f
   \cdot\overline{f}) \right)^2\\&+\frac p2|f|^{  p-2  }   \sum_{j=1}^{2n_\alpha} \left( X_{\alpha j} ^2 f  \cdot \overline{f} + f   \cdot \overline{X_{\alpha j}
   ^2
   f }\right)+   p |f|^{  p-2  }\sum_{j=1}^{2n_\alpha}|X_{\alpha j} f|^2 .
   \end{split} \end{equation}\eqref{eq:p-subharmonic00}   minus
   \begin{equation}\label{eq:p-subharmonic0}\begin{split}  \partial_{t_\alpha} |f|^p&=\frac p2 |f|^{  p-2  } \left(
\partial_{t_\alpha} f \cdot\overline{f} +  f \cdot\overline{\partial_{t_\alpha} f}\right ),
  \end{split} \end{equation} multiplied by $4n_\alpha$ gives us
  \begin{equation}\label{eq:p-subharmonic}\begin{split}
   4n_\alpha \mathcal{ L}_\alpha |f|^p=&  p \left(  p-2 \right) |f|^{  p-4  }  \sum_{j=1}^{2n_\alpha}  \big({\rm Re} (X_{\alpha j} f\cdot \overline{f})
   \big)^2  + p  |f|^{  p-2  } \sum_{j=1}^{2n_\alpha} |X_{\alpha j} f|^2
   \end{split} \end{equation}by  $4n_\alpha\mathcal{L}_\alpha f=\sum_{j=1}^{2n_\alpha}   X_{\alpha j} ^2 f-4n_\alpha\partial_{t_\alpha} f=0$ by Theorem
   \ref{prop:regular-sublap}.

Recall that for a holomorphic function $u$ on a domain $\Omega\subset \mathbb{C}$, $\ln |u |$ and so $|u|^p$ for $ p >0 $ are subharmonic on   $\Omega $.
To apply this property, we consider
  $F:=(\pi_\alpha^{-1}) ^*f$, which is  holomorphic  on $  {\mathcal U}_\alpha  $ by Proposition \ref{prop:regular-equiv}. Then,   $ |F|^p $ is plurisubharmonic.
 Thus
  \begin{equation}\label{eq:subharmonic-p}
     \sum_{j=1}^{2n_\alpha} \frac {\partial^2|F|^p  }{\partial {\widetilde  x_{\alpha j}}^2}(\widetilde  w_\alpha, \widetilde   {\mathbf{z}}_\alpha)\geq 0,
  \end{equation}when $F(\widetilde  w_\alpha, \widetilde   {\mathbf{z}}_\alpha)\neq 0$.
Similarly as in \eqref{eq:p-subharmonic00},
this subharmonicity   implies
  \begin{equation}\label{eq:p-subharmonic2}\begin{split}
   0\leq  \sum_{j=1}^{2n_\alpha}  \frac {\partial^2|F|^p  }{\partial {\widetilde  x_{\alpha j}}^2}=&  p \left(  p-2 \right) |F|^{  p-4  }
   \sum_{j=1}^{2n_\alpha}  \left({\rm Re}\left (\frac {\partial F   }{\partial {\widetilde  x_{\alpha j}} } \cdot \overline{F}\right) \right)^2  +  p  |F|^{  p-2 }
   \sum_{j=1}^{2n_\alpha} \left| \frac {\partial F   }{\partial {\widetilde  x_{\alpha j}} }\right |^2,
 \end{split} \end{equation}
 by using
 \begin{equation*}
      \frac {\partial^2 F   }{\partial {\widetilde  x_{\alpha j}}^2}+\frac {\partial^2 F  \hskip 8mm }{\partial {\widetilde  x_{\alpha(n_\alpha + j)}}^2}=  0,
 \end{equation*}$j=1,\ldots, n_\alpha$, since  $F$ is holomorphic.
Apply (\ref{eq:p-subharmonic2}) to   (\ref{eq:p-subharmonic})  to get $\mathcal{ L}_\alpha |f|^p({t}_\alpha ,\mathbf{0}_\alpha)\geq 0$   by
\begin{equation*}X_{\alpha j} f({t}_\alpha ,\mathbf{0}_\alpha)=\frac {\partial F   }{\partial {\widetilde  x_{\alpha j}} }({t}_\alpha ,\mathbf{0}_\alpha) ,
\end{equation*}
  since
  $
        \pi_{\alpha*} X_{\alpha  j}|_{( {t}_\alpha ,\mathbf{0}_\alpha)}= \frac {\partial    }{\partial {\widetilde  x_{\alpha j}} }|_{( {t}_\alpha ,\mathbf{0}_\alpha )}
  $
by \eqref{eq:pi-XY}.
The proof of Proposition \ref{prop:sub}   is completed.
  \end{proof}
 \begin{cor}  \label{cor:sub}  Suppose $f$ is holomorphic on $\mathscr U $. Then for any  $  p>0$, we have
 \begin{equation}\label{eq:p-subparabolic}
   \mathcal{ L}_\alpha|f|^p( \mathbf{t},\mathbf{g}  )\geq 0,
 \end{equation}
 for $( \mathbf{t},\mathbf{g } )\in\mathscr U  $ where $f( \mathbf{t},\mathbf{g } )\neq 0$.
\end{cor}
 \subsection{Proof of    Proposition  \ref{prop:heat-kernel-Hp}}
  Maximum principle can not be applied to $|f|^q$  directly because it is not  smooth on $|f|=0$. It is not easy to construct an auxiliary   function as in
  \cite[Section 3.2.1 in Chapter 7]{St70} to overcome this difficulty. But the argument of the proof of
  maximum principle can   be easily adapted to  this case as follows. Let $ f_k(\mathbf{{t}}  ,\mathbf{g}): =f ( \mathbf{{t}} +{\boldsymbol {\varepsilon }}_k ,\mathbf{g})$ as before, and let
  \begin{equation}\label{eq:v-t2=0}\begin{split}
    v_k (\mathbf{t} ,\mathbf{g}):&=|f_k(\mathbf{t} ,\mathbf{g})|^q-\kappa t_1-\kappa t_2,
 \\
    \widetilde{v}_k ( \mathbf{t}, \mathbf{g}):&=[ h_\mathbf{t}* v_k (\mathbf{0} , \cdot)]  (\mathbf{g}),
  \end{split}\end{equation} for $\kappa>0$.
 Here $v_k (\mathbf{0} , \cdot)=|f_k(\mathbf{0}, \cdot)|^q\in L^{\tilde q}(\mathscr H_1\times \mathscr H_2 )$ with ${\tilde q}=1/q>1$. Thus, $\widetilde{v}_k ( \mathbf{t}, \mathbf{g})$ is smooth on $\overline{\mathscr U}$,
  \begin{equation*}v_k (\mathbf{0} , \cdot)=\widetilde{v}_k (\mathbf{0} , \cdot)  \qquad \text{  and }\qquad
     \mathcal{L}_\alpha\widetilde{v}_k( \mathbf{t}, \mathbf{g}) =0 \quad \text{  for }\quad  ( \mathbf{t}, \mathbf{g})\in \mathscr U.
  \end{equation*}

  By Fubini's theorem, $v_k (\mathbf{0 },  \mathbf{{g }}_1,\mathbf{g}_2)=|f_k(\mathbf{0},  \mathbf{{g }}_1,\mathbf{g}_2)|^q$ belongs to $L^{\tilde q}(\mathscr H_1 )$ for
  almost all $\mathbf{g}_2\in\mathscr H_2$.
  Now we fix such a $\mathbf{g}_2\in \mathscr H_2$, and denote functions on $\mathscr U_1$:
   \begin{equation*}\begin{split}
       v_{ k (0,  \mathbf{g}_2)}(t_1,\mathbf{g}_1):&=v_k ( t_1 ,  0, \mathbf{g}_1,\mathbf{g}_2) ,\\  \widetilde v_{ k (0,  \mathbf{g}_2)}  (t_1,\mathbf{g}_1):&=\widetilde{v}_k ( t_1 , 0,
      \mathbf{g}_1,\mathbf{g}_2).
   \end{split}  \end{equation*}

  As in (\ref{eq:claim})  in the proof of Proposition \ref{prop:heat-kernel-H1},   for given $T>0$ and $\eta >0$, there
  exists $r_0>0$ such that $ |  v_{ k (0,  \mathbf{g}_2)} - \widetilde v_{ k (0,  \mathbf{g}_2)} |\leq \eta $ on  the boundary  $ [0,T) \times\partial B_1(\mathbf{0}_1,r) $ for $r\geq r_0$. Then,
 \begin{equation}\label{eq:L-+}
    \mathcal{L}_1\left[ v_{ k (0,  \mathbf{g}_2)} - \widetilde v_{k (0,  \mathbf{g}_2)} -2 \eta\right](t_1,\mathbf{g}_1)>0
 \end{equation} when $ {v}_k   (\mathbf{t},\mathbf{g} )\neq 0$,
 by $\mathcal{L}_1 (-\kappa t_1) =\kappa>0$. Moreover,  $ v_ {k ( 0,  \mathbf{g}_2)} - \widetilde v_ { k(0,  \mathbf{g}_2)} - 2\eta$ is negative on the boundary $ [0,T) \times\partial
  B_1(\mathbf{0}_1,r) \cup \{0\}\times B_1(\mathbf{0}_1,r)$.

Suppose that $\left[ v_{ k (0,  \mathbf{g}_2)} - \widetilde v_{ k (0,  \mathbf{g}_2)} -2 \eta\right]({t}_1,  \mathbf{{g }}_1  )\geq 0$ at some point in  $({t}_1,  \mathbf{{g }}_1   ) \in (0,T) \times
B_1(\mathbf{0}_1,r)$. Then, argued  as in the proof of Proposition
\ref{prop:max}, we can find $(t_1^*,\mathbf{g}_1^* ) \in  (0,T) \times B_1(\mathbf{0}_1,r)$
such that
 \begin{equation}\label{eq:0-point}\begin{split} \left[ v_{ k (0,  \mathbf{g}_2)} - \widetilde v_ {k (0,  \mathbf{g}_2)} - 2\eta\right](t_1^*, \mathbf{g}_1^*  )=0, \qquad\left[ v_ { k(0,  \mathbf{g}_2)} - \widetilde v_{ k (0,  \mathbf{g}_2)} -2 \eta\right](t_1, \mathbf{g}_1   )<0,  \end{split}  \end{equation}
  for
$0<t_1<t_1^*$, $\mathbf{g}_1\in B_1(\mathbf{0}_1,r)$.
Note that   we must have $ f_k(t_1^*,0,  \mathbf{g}_1^* , \mathbf{g}_2 )\neq0$. Otherwise, we have $   v_{k (0,  \mathbf{g}_2)}(t_1^*, \mathbf{g}_1^*  )< 0$ by definition \eqref{eq:v-t2=0} and
\begin{equation}\label{eq:0-point-2}
   \left[ v_{k (0,  \mathbf{g}_2)} - \widetilde v_{k (0,  \mathbf{g}_2)} - 2\eta\right](t_1^*, \mathbf{g}_1^*
 )<0.
\end{equation}which is contradict to \eqref{eq:0-point}.
\eqref{eq:0-point-2} holds because $ \widetilde v_{k (0,  \mathbf{g}_2)}(t_1^*, \mathbf{g}_1^*  )\geq 0$
  by the third formula in \eqref{eq:v-t2=0} by $v_k (\mathbf{0} , {\mathbf{g}}  ) =  |f_k  (\mathbf{0} ,  {\mathbf{g}} )|^q\geq0$ and the nonnegativity
of the heat kernel \cite[Proposition 1.68]{FS}. Therefore,  $  v_{ k (0,  \mathbf{g}_2)}  $ and so  $  v_{ k (0,  \mathbf{g}_2)} - \widetilde v_ { k(0,  \mathbf{g}_2)} - 2\eta $ is smooth at $(t_1^*, \mathbf{g}_1^*  )$. As in \eqref{eq:maximum-1}-\eqref{eq:maximum-2},  we get
\begin{equation*}
   \mathcal{L}_1\left[ v_{k (0,  \mathbf{g}_2)} - \widetilde v_{ k (0,  \mathbf{g}_2)} - 2\eta\right](t_1^*, \mathbf{g}_1^* )\leq0,
\end{equation*}
  which
contradicts to (\ref{eq:L-+}). Thus $\left[ v_ {k (0,  \mathbf{g}_2)} - \widetilde v_ {k (0,  \mathbf{g}_2)} -2 \eta\right](t_1,  \mathbf{g}_1  )<0$  for $(  {t}_1,  \mathbf{{g }}_1) \in (0,T) \times
B_1(\mathbf{0}_1,r)$ for any fixed $\eta, T >0$. Letting  $r\rightarrow\infty$, $T\rightarrow\infty$ and $\eta\rightarrow 0$, we get
$
    {v}_k (  {t}_1,0, \mathbf{{g }}  )\leq \widetilde{v}_k (  {t}_1,0, \mathbf{{g }} )
$ for $ (  {t}_1,  \mathbf{{g }}_1) \in\mathscr U_1$ and almost all $\mathbf{g}_2$, and so for all $\mathbf{g}_2\in  \mathscr H_2$ by continuity.

Applying the same argument to
\begin{equation*}\begin{split}
       {v}_{k ({t}_1,  \mathbf{{g }}_1 )} ( t_2, \mathbf{g}_2 ):&={v}_k (  {t}_1, t_2, \mathbf{{g }}_1,\mathbf{g}_2 ),\\
       \widetilde{{v}}_{k ({t}_1,  \mathbf{{g }}_1 )} ( t_2, \mathbf{g}_2 ):&=\widetilde{v}_k (  {t}_1,t_2, \mathbf{{g }}_1,\mathbf{g}_2 ),
   \end{split}  \end{equation*}
  as functions on $\mathscr U_2$ for fixed
$ t_1  ,\mathbf{g}_1$, we get $ {v}_ {k({t}_1,  \mathbf{{g }}_1 )}  \leq
        \widetilde{{v}}_ {k({t}_1,  \mathbf{{g }}_1 )}$. Consequently,
\begin{equation}\label{eq:v<hat-v}
     |f_k( \mathbf{t}, \mathbf{g})|^q \leq \int_{\mathscr H_1\times \mathscr H_2 }h_\mathbf{t}\left(\mathbf{h}^{-1} \mathbf{g}\right)|f_k (\mathbf{0 } ,\mathbf{h})|^q d\mathbf{h}
 \end{equation} for any $( \mathbf{t}, \mathbf{g})\in \mathscr U  $, by letting $\kappa\rightarrow 0$.

Since $ |f(\boldsymbol {\varepsilon }_k,\cdot)|^q\in L^{\tilde q}(\mathscr H_1\times \mathscr H_2 )$ with ${\tilde q}=1/q>1$ and $   L^{\tilde q}(\mathscr H_1\times \mathscr H_2)$ is reflexive, there exists a
subsequence weakly convergent to some $\widetilde{f} \in L^{\tilde q}(\mathscr H_1\times \mathscr H_2)$
by Banach-Alaoglu theorem. We must have $\widetilde{f}=|f(\mathbf{0 },\cdot)|^q $ by the continuity of $f$ on $\overline{\mathscr U} $. Taking limit in
\eqref{eq:v<hat-v}, we get the inequality \eqref{eq:heat-rep-p}.
  \qed
\vskip 3mm
{\bf Funding}.  The first author is  partially  supported by National Nature Science
Foundation
  of China  (Nos. 12171221, 12071197) and Natural Science Foundation of Shandong Province
(Nos. ZR2021MA031, 2020KJI002); the second author is  partially  supported by National Nature Science
Foundation of China (No.
11971425).

\end{document}